\documentclass[11pt]{amsart}

\usepackage{amsmath, amssymb, amsthm, setspace, tikz-cd, hyperref}
\usepackage[shortcuts]{extdash}
\usepackage[all]{xy}
\theoremstyle{plain}
\newtheorem{theorem}{Theorem}[section]
\newtheorem*{theorem*}{Theorem}
\newtheorem*{MCT}{Monotone Convergence Theorem}
\newtheorem*{corollary*}{Corollary}
\newtheorem{lemma}[theorem]{Lemma}
\newtheorem{proposition}[theorem]{Proposition}
\newtheorem{corollary}[theorem]{Corollary}

\newtheorem{conjecture}[theorem]{Conjecture}

\theoremstyle{definition}
\newtheorem{definition}[theorem]{Definition}
\newtheorem{exampletemp}[theorem]{Example}
\newtheorem{remark}[theorem]{Remark}

\newcommand{\cstar}{\ensuremath{\mathrm{C}^{*}}}
\newcommand{\reg}{\mathrm{reg}}
\newcommand{\sng}{\mathrm{sng}}

\begin{document}

\title{Traces on Topological Graph Algebras}
\author{Christopher Schafhauser}
\address{Pure Mathematics, University of Waterloo, 200 University Avenue West, Waterloo, ON, Canada, N2L 3G1}
\email{cschafhauser@uwaterloo.ca}
\date{\today}

\begin{abstract}
  Given a topological graph $E$, we give a complete description of tracial states on the \cstar-algebra $\cstar(E)$ which are invariant under the gauge action; there is an affine homeomorphism between the space of gauge invariant tracial states on $\cstar(E)$ and Radon probability measures on the vertex space $E^0$ which are, in a suitable sense, invariant under the action of the edge space $E^1$.  It is shown that if $E$ has no cycles, then every tracial state on $\cstar(E)$ is gauge invariant.  When $E^0$ is totally disconnected, the gauge invariant tracial states on $\cstar(E)$ are in bijection with the states on $\mathrm{K}_0(\cstar(E))$.
\end{abstract}

\maketitle

\section{Introduction}

The class of topological graph algebras was introduced by Katsura in \cite{Katsura:TGA1} and further developed in \cite{Katsura:TGA2, Katsura:TGA3, Katsura:TGA4}.  These \cstar-algebras provide a simultaneous generalization of both homeomorphism algebras and the Cuntz-Kreiger graph algebras.  This class of algebras contains a large class of nuclear \cstar-algebras including all AF-algebras, Kirchberg algebras satisfying the Universal Coefficient Theorem (UCT), and separable, simple, unital, real rank zero A$\mathbb{T}$-algebras.  A large collection of examples is given in \cite{Katsura:TGA2}.  In fact, there is currently no example of a nuclear \cstar-algebra satisfying the UCT which is known not to be defined by a topological graph, although R{\o}rdam's examples of simple, nuclear, \cstar-algebras with finite and infinite projections constructed in \cite{Rordam} are likely candidates.

In this paper, we consider the tracial states on topological graph algebras.  Given a homeomorphism $\sigma$ of a compact metric space $X$, it is a standard fact that every Radon probability measure $\mu$ on $X$ which is invariant under $\sigma$ induces a tracial state on the crossed product $C(X) \rtimes_\sigma \mathbb{Z}$.  For the Cuntz-Krieger graph algebras generated by a discrete graph, this problem was studied by Tomforde in \cite{Tomforde:OrderedKTheory}.  Tomforde introduced the notion a \emph{graph trace} which is a probability measure on the vertex set $E^0$ of the graph $E$ which is, in a suitable sense, invariant under the action of the edges $E^1$.  We provide a simultaneous generalization of these two theorems.  Below is a summary of our results.

\begin{theorem*}
Let $E$ be a topological graph.
\begin{enumerate}
  \item There is an affine homeomorphism between the space of invariant measures on $E$ and the space of gauge invariant tracial states on $\cstar(E)$.
  \item If $E^0$ is totally disconnected, there is a affine homeomorphism between the space of invariant measures on $E$ and the space of states on $\mathrm{K}_0(\cstar(E))$.
  \item If $E$ is has no cycles, every tracial state on $\cstar(E)$ is gauge invariant.
\end{enumerate}
\end{theorem*}

Given a tracial state $\tau$ on $\cstar(E)$, the composition
\[ \begin{tikzcd} C_0(E^0) \arrow{r}{\iota} & \cstar(E) \arrow{r}{\tau} & \mathbb{C} \end{tikzcd} \]
is a state on $C_0(E^0)$ and hence is given by integration against a Radon probability measure $\mu$ on $E^0$.  The invariance of $\mu$ can be verified directly using the Cuntz-Krieger relations in $\cstar(E)$.

Most of the work involved is in proving the converse.  The key observation is that given a tracial state $\tau$ on $\cstar(E)$, the GNS space $L^2(\cstar(E), \tau)$ decomposes as a direct integral over the boundary path space $\partial E$ developed by Yeend in \cite{Yeend1} and \cite{Yeend2}.  The fibres of the direct integral and the measure on the base space $\partial E$ can be described explicitly in terms of the graph $E$ and the invariant measure $\mu$ on $E$.  Hence it is possible to build the GNS representation of the desired tracial state using only the invariant measure on the graph and then the trace is constructed as a vector state on this space.

More precisely, given a topological graph $E$, we define a Hilbert module $L^2(E)$ over the algebra $C_0(\partial E)$ and a representation $\pi : \cstar(E) \rightarrow \mathbb{B}(L^2(E))$.  Given an invariant measure $\mu$ on the graph $E$, there is an induced measure $\tilde{\mu}$ on $\partial E$ and extending scalars yields a Hilbert space
\[ L^2(E, \mu) := L^2(E) \otimes_{C_0(\partial E)} L^2(\partial E, \mu) \]
and a representation $\pi_\mu$ of $\cstar(E)$ on the space $L^2(E, \mu)$.  Moreover, the representation $\pi_\mu$ admits a cyclic vector $\hat{u} \in L^2(E, \mu)$ such that $\tau := \langle \hat{u}, \pi( \cdot ) \hat{u} \rangle$  is a tracial state on $\cstar(E)$.  This trace is indeed invariant under the gauge action and recovers the given measure $\mu$ when restricted to $C_0(E^0)$.

Describing all tracial states on $\cstar(E)$ would be a difficult task in general.  However, it appears that every tracial state on $\cstar(E)$ will be gauge invariant in many important cases.  For homeomorphism algebras, this is the case whenever the action is free and for discrete graphs, this is the case when the graph satisfies condition (K).  In \cite{Katsura:TGA3}, Katsura introduced a freeness condition for topological graphs which yields a simultaneous generalization of both these notions.  It is probably true that every tracial state on a topological graph algebra $\cstar(E)$ is gauge invariant whenever $E$ is free and, in particular, whenever $\cstar(E)$ is simple.  We were not able to prove this in full generality, but the result holds in the special case when $E$ has no cycles.

In the case when $E$ is a totally disconnected topological graph, the gauge invariant tracial states on $\cstar(E)$ are precisely those which can be detected by the ordered group $\mathrm{K}_0(\cstar(E))$.  As an abelian group, $\mathrm{K}_0(\cstar(E))$ can be described explicitly via the Pimsner-Voiculescu sequence for topological graph algebras, but very little is known about the order structure in general.  When $E$ is a minimal, totally disconnected graph and $E^0$ is compact, our results on traces can be used, at least in principle, to describe the positive cone in $\mathrm{K}_0(\cstar(E))$ up to perforation.  It is likely that $\mathrm{K}_0(\cstar(E))$ is weakly unperforated for minimal, totally disconnected graphs $E$ which would then give a complete description of the positive cone on $\mathrm{K}_0(\cstar(E))$.

The paper is structured as follows.  Section 2 recalls a few definitions and fundamental results on topological graphs in order to standardize notation.  Section 3 is devoted to the boundary path space $\partial E$ of $E$.  Invariant measures on topological graphs are introduced in Section 4 and it is shown that invariant measures on $E$ extend to invariant measures on $\partial E$.  In section 5, the representations of $\cstar(E)$ on the Hilbert module $L^2(E)$ and the Hilbert spaces $L^2(E, \mu)$ mentioned above are constructed.  Section 6 contains most of our main results and the final section is devoted to totally disconnected graphs and $\mathrm{K}$-theoretic calculations.

We end this introduction with some remarks on notation.  On Hilbert modules and Hilbert spaces, inner products are always linear in the second variable and conjugate linear in the first.  For graph algebras, there are two competing conventions in the literature.  We follow the conventions in Raeburn's book \cite{Raeburn}, which agrees with the notation used by Katsura in \cite{Katsura:TGA1} where topological graphs were introduced; in particular, an edge represents a partial isometry from the projection at the source of the edge into the projection at the range of the edge.

We do not assume our spaces are second countable.  Because of this, there are some measure theoretic technicalities which occur.  The following extension of the Monotone Convergence Theorem (see Proposition 7.12 in \cite{Folland}) allows us to work around these issues in the non-separable case.  It will be used implicitly throughout the paper.

\begin{MCT}
Suppose $X$ is a locally compact Hausdorff space and $\mu$ is a Radon measure on $X$.  If $(f_n) \subseteq C_c(X)$ is an increasing net of positive functions converging pointwise to a (necessarily Borel) function $f : X \rightarrow [0, \infty]$, then
\[ \int_X f \, d\mu = \lim_n \int_X f_n \, d\mu. \]
\end{MCT}

\section{Topological Graph Algebras}

We recall briefly the definition of topological graphs and the \cstar-algebras they generate in order to standardize notation.  The reader is referred to Katsura's paper \cite{Katsura:TGA1} for details.

\begin{definition}
A \emph{topological graph} $E = (E^0, E^1, r, s)$ consists of locally compact Hausdorff spaces $E^0$ and $E^1$, a continuous function $r : E^1 \rightarrow E^0$, and a local homeomorphism $s : E^1 \rightarrow E^0$.  Elements of $E^0$ and $E^1$ are called \emph{vertices} and \emph{edges} respectively.
\end{definition}

There are two subsets of the vertex space $E^0$ which have a special role in the theory.

\begin{definition}
There is a maximal open set $E^0_\reg \subseteq E^0$ such that $r$ restricts to a proper surjection $r^{-1}(E^0_\reg) \rightarrow E^0_\reg$.  Define $E^0_\sng = E^0 \setminus E^0_\reg$.  The elements of $E^0_\reg$ and $E^0_\sng$ are called \emph{regular} and \emph{singular} vertices, respectively.
\end{definition}

\begin{definition}
A \emph{path of length $n$} in $E$ consists of a word $\alpha_1 \cdots \alpha_n$ such that $\alpha_i \in E^1$ and $s(\alpha_i) = r(\alpha_{i+1})$ for each $i = 1, \ldots, n-1$.  Let $E^n$ denote the path of lengths $n$.  Define the \emph{finite path space} of $E$ by $E^* = \coprod_{n=0}^\infty E^n$.  We endow $E^n$ with relative product topology and $E^*$ with the topology generated by the $E^n$.  Extend the range and source maps to $r, s : E^* \rightarrow E^0$ by
\[ r(\alpha_1 \cdots \alpha_n) = r(\alpha_1) \quad \text{and} \quad s(\alpha_1 \cdots \alpha_n) = s(\alpha_n). \]
Then $r : E^* \rightarrow E^0$ is continuous and $s : E^* \rightarrow E^0$ is a local homeomorphism.

Define the \emph{infinite path space} by
\[ E^\infty = \{ \alpha_1 \alpha_2 \cdots : s(\alpha_i) = r(\alpha_{i+1})  \text{ for every $i \in \mathbb{N}$} \} \]
Extend the range map to $r : E^\infty \rightarrow E^0$ by $r(\alpha_1 \alpha_2 \cdots ) = r(\alpha_1)$.

Given a path $\alpha \in E^* \cup E^\infty$, we write $|\alpha|$ for the length of $\alpha$.  That is, $|\alpha| = n$ if $\alpha \in E^n$ for $n \in \mathbb{N}$ and $|\alpha| = \infty$ if $\alpha \in E^\infty$.
\end{definition}

The most obvious choice of topology on $E^\infty$ is the product topology.  The defect is that the product topology on $E^\infty$ need not be locally compact.  Yeend introduced a larger space $\partial E$ containing $E^\infty$ and showed there is a natural locally compact topology on $\partial E$.  See Section \ref{sec:BoundaryPathSpace} below.

For our purposes, it will be useful to partition the finite path space $E^*$ into two subsets: paths starting at regular vertices and paths starting at singular vertices.  We introduce the following notation.

\begin{definition}\label{defn:RegularPaths}
Define
\[ E^n_\sng = \{ \alpha \in E^n : s(\alpha) \in E^0_\sng \} \quad \text{and} \quad E^n_\reg = \{ \alpha \in E^n : s(\alpha) \in E^0_\reg \}. \]
The sets $E^*_\sng$ and $E^*_\reg$ are defined similarly.
\end{definition}

We now define a \cstar-algebra $\cstar(E)$ associated to a topological graph $E$.  First we recall a simple fact which will be need throughout the paper.  The proof follows as in Lemmas 1.4 and 1.5 of \cite{Katsura:TGA1}.

\begin{lemma}\label{lem:LocalHomeomorphism}
Suppose $X$ and $Y$ are locally compact Hausdorff spaces and $\varphi : X \rightarrow Y$ is a local homoeomorphism.  If $f \in C_c(X)$, the function
\[ Y \rightarrow \mathbb{C}, \qquad y \mapsto \sum_{x \in \varphi^{-1}(y)} f(x) \]
is defined and is compactly supported and continuous.  In fact, the sum occurring on the right is always finite.
\end{lemma}

Given a topological graph $E$, the space $C_c(E^1)$ is a bimodule over $C_0(E^0)$ with the left and right actions given by $a \cdot \xi = (a \circ r) \xi$ and $\xi \cdot a = \xi(a \circ s)$ for $a \in C_0(E^0)$ and $\xi \in C_c(E^1)$.  Moreover, there is a $C_0(E^0)$-valued inner product on $C_c(E^1)$ given by
\[ \langle \xi, \eta \rangle (v) = \sum_{e \in s^{-1}(v)} \overline{\xi(e)} \eta(e) \]
for each $\xi, \eta \in C_c(E^1)$ and $v \in E^0$.

\begin{definition}\label{defn:GraphCorrespondence}
Let $H(E)$ denote the \cstar-correspondence obtained as the completion of $C_c(E^1)$ with respect to the inner product given above.
\end{definition}

The topological graph algebra $\cstar(E)$ is defined as the Cuntz-Pimsner algebra of the \cstar-correspondence $(H(E), C_0(E^0))$.  We recall briefly how this is defined.

\begin{definition}
A \emph{Toeplitz representation} $(\pi^0, \pi^1)$ of $E$ on a \cstar-algebra $A$ consists of a $\ast$-homomorphism $\pi^0 : C_0(E^0) \rightarrow A$ and a linear map $\pi^1 : H(E) \rightarrow A$ such that
\[ \pi^1(\xi)^*\pi^1(\eta) = \pi^0(\langle \xi, \eta \rangle) \qquad \text{and} \qquad \pi^0(a) \pi^1(\xi) = \pi^1(a\xi) \]
for all $a \in C_0(E^0)$ and $\xi, \eta \in H(E)$.
\end{definition}

Given a Toeplitz representation $(\pi^0, \pi^1)$ of $E$ on $A$, for each $k \geq 2$, there is an induced linear map $\pi^k : H(E)^{\otimes k} \rightarrow A$ given by
\[ \pi^k(\xi_1 \otimes \ldots \otimes \xi_k) = \pi^1(\xi_1) \cdots \pi^1(\xi_k). \]
For any $k \in \mathbb{N}$, $C_c(E^k)$ is a dense submodule of the tensor power $H(E)^{\otimes k}$; given functions $\xi_1, \ldots, \xi_k \in C_c(E^1)$, the elementary tensor $\xi_1 \otimes \cdots \otimes \xi_k$ acts on $E^k$ by
\[ (\xi_1 \otimes \cdots \otimes \xi_k)(\alpha_1 \cdots \alpha_k) = \xi_1(\alpha_1) \cdots \xi_k(\alpha_k) \]
for all $\alpha_1 \cdots \alpha_k \in E^k$.

For any locally compact Hausdorff space $X$, $C_b(X)$ denotes the \cstar-algebra of continuous bounded functions on $X$.

\begin{proposition}\label{prop:LeftMultiplicationMap}
There is a $\ast$-homomorphism $\lambda : C_b(E^1) \rightarrow \mathbb{B}(H(E))$ given by pointwise multiplication.  Moreover, $\lambda(\xi) \in \mathbb{K}(H(E))$ if and only if $\xi \in C_0(E^1)$.
\end{proposition}

For any Hilbert module $H$ and $\xi, \eta \in H$, let $\xi \otimes \eta^*$ denote the rank one operator on $H$ defined by $(\xi \otimes \eta^*)(\zeta) = \xi \langle \eta, \zeta \rangle$ for $\zeta \in H$.

\begin{proposition}\label{prop:InducedMapOnCompacts}
Given a Toeplitz representation $(\pi^0, \pi^1)$ of a topological graph $E$ on a $\cstar$-algebra $A$, there is an induced $\ast$-homomorphism $\varphi : \mathbb{K}(H(E)) \rightarrow A$ determined by $\varphi(\xi \otimes \eta^*) = \pi^1(\xi) \pi^1(\eta)^*$.
\end{proposition}

\begin{definition}\label{defn:CovariantRepresentation}
A Toeplitz representation $(\pi^0, \pi^1)$ of $E$ is called a \emph{covariant representation} of $E$ if $\varphi(\lambda(a \circ r)) = \pi^0(a)$ for all $a \in C_0(E^0_\reg)$.  The \emph{topological graph algebra} $\cstar(E)$ is defined as the universal \cstar-algebra generated by a covariant representation of $E$.
\end{definition}

For a topological graph $E$, there is a natural action $\gamma : \mathbb{T} \curvearrowright \cstar(E)$ which we now describe.  If $\pi^0 : C_0(E^0) \rightarrow \cstar(E)$ and $\pi^1 : H(E) \rightarrow \cstar(E)$ are the canonical maps, the action $\gamma$ is determined by
\[ \gamma_z(\pi^0(a)) = \pi^0(a) \qquad \text{and} \qquad \gamma_z(\pi^1(\xi)) = z \pi^1(\xi) \]
for each $a \in C_0(E^0)$, $\xi \in H(E)$, and $z \in \mathbb{T}$.  The action $\gamma$ is called the \emph{gauge action}.

We end this section with two technical lemmas which will be needed later.  The first follows by combining Lemmas 1.15 and 1.16 in \cite{Katsura:TGA1}.

\begin{lemma}\label{lem:PartitionOfUnity}
If $\zeta \in C_c(E^1)$, there are $n \in \mathbb{N}$ and $\xi_i, \eta_i \in C_c(E^1)$ for $i = 1, \ldots, n$, such that
\begin{enumerate}
  \item $\zeta = \sum_{i=1}^n \xi_i \overline{\eta_i}$,
  \item $\lambda(\zeta) = \sum_{i=1}^n \xi_i \otimes \eta_i^*$, and
  \item if $e, e' \in E^1$ are distinct edges with $s(e) = s(e')$, then $\xi_i(e) \eta_i(e') = 0$.
\end{enumerate}
\end{lemma}

\begin{lemma}\label{lem:InequalityForFunctionsOnEdges}
Let $(\pi^0, \pi^1)$ be a Toeplitz representation of a topological graph $E$ and let $\varphi$ denote the induced representation of $\mathbb{K}(H(E))$.  If $\zeta \in C_c(E^1)$ and $a \in C_c(E^0)$ with $0 \leq \zeta \leq a \circ r$, then $0 \leq \varphi(\lambda(\zeta)) \leq \pi^0(a)$.
\end{lemma}

\begin{proof}
Let $H$ denote the Fock space associated to $H(E)$.  The Toeplitz representation $(\pi^0, \pi^1)$ of $E$ on $\mathbb{B}(H)$ defined by
\[ \pi^0(a)(\phi) = a \phi \qquad \text{and} \qquad \pi^1(\xi)(\phi) = \xi \otimes \phi \]
is the universal Toeplitz representation of $E$.  Hence it is enough to verify the lemma for this representation.

Fix $\xi_i, \eta_i \in C_c(E^1)$ for $1 \leq i \leq n$ as in Lemma \ref{lem:PartitionOfUnity}.  Given $\phi \in H(E)$, $\phi' \in H$, note that
\begin{align*}
  \langle \varphi(\lambda(\zeta)) (\phi \otimes \phi'), \phi \otimes \phi' \rangle &=  \sum_i \langle (\xi_i \otimes \eta_i^*)(\phi) \otimes \phi', \phi \otimes \phi' \rangle = \langle \lambda(\zeta)(\phi) \otimes \phi', \phi \otimes \phi' \rangle \\
  &\leq \langle \lambda(a \circ r)(\phi) \otimes \phi', \phi \otimes \phi' \rangle = \langle \pi^0(a) (\phi \otimes \phi'), \phi \otimes \phi' \rangle.
\end{align*}
Moreover, when $\phi \in H^{\otimes 0} = C_0(E^0)$, we have
\[ \langle \varphi(\lambda(\zeta)) (\phi), \phi \rangle = 0 \leq \langle \pi^0(a) (\phi), \phi \rangle. \]
It follows that $\langle \varphi(\lambda(\zeta)) \phi, \phi \rangle \leq \langle \pi^0(a) \phi, \phi \rangle$ for all $\phi \in H$ and hence $\varphi(\lambda(\zeta)) \leq \pi^0(a)$.
\end{proof}

\section{The Boundary Path Space}\label{sec:BoundaryPathSpace}

In \cite{Yeend1} and \cite{Yeend2}, Yeend introduced a locally compact Hausdorff space $\partial E$ associated to a (higher rank) topological graph $E$ called the \emph{boundary path space}.  In the case of a rank one topological graph $E$ with $E^0_\reg = E^0$, $\partial E$ is exactly the infinite path space $E^\infty$ with the product topology.  In general, the product topology on $E^\infty$ will not be locally compact and the space $\partial E$ is the appropriate replacement.

Kumjian and Li gave a simpler description of $\partial E$ for rank one topological graphs in \cite{KumjianLi} (see also \cite{Webster} for the case when $E$ is discrete).  We follow their approach.  For our purposes, it will be useful to have approximate versions $\partial E_n$ of the boundary path space whose projective limit is $\partial E$ (see Proposition \ref{prop:BoundaryProjectiveLimit}).

Recall from Definition \ref{defn:RegularPaths}, $E^*_\sng$ and $E^*_\reg$ denote the set of paths whose source is a singular or regular vertex, respectively.  Define
\[ \partial E_n = E^0_\text{sng} \cup \cdots \cup E^{n-1}_\sng \cup E^n \quad \text{and} \quad \partial E = E^*_\sng \cup E^\infty. \]
For $S \subseteq E^*$ and $n \in \mathbb{N}$, define
\[ Z_n(S) = \{ \alpha \in  \partial E_n : r(\alpha) \in S \text{ or } \alpha_1 \cdots \alpha_k \in S \text{ for some  } 1 \leq k \leq n \} \subseteq \partial E_n \]
and define $Z(S) \subseteq \partial E$ similarly.

The spaces $\partial E$ and $\partial E_n$ carry natural locally compact Hausdorff topologies (Proposition \ref{prop:BoundaryIsLCH}).  This was shown by Yeend in \cite{Yeend1} for higher rank topological graphs using a different description of $\partial E$.  That our definition agrees with Yeend's is Lemma 4.5 in \cite{KumjianLi}.  The result for $\partial E_n$ has an identical proof.  We give a careful description of these topologies and a proof that they are locally compact and Hausdorff since the proof is not easy to extract from \cite{Yeend1} and the details are omitted in \cite{KumjianLi}.

First we need a topological lemma.  The proof is an easy exercise in constructing subnets.

\begin{lemma}\label{lem:subnet}
Suppose $X$ is a topological space, $(x_i)_{i \in I}$ is a net in $X$, and $x \in X$.  If for every neighborhood $U$ of $x$, the set $\{ i \in I : x_i \in U \}$ is cofinal in $I$, then $(x_i)$ has a subnet converging to $x$.
\end{lemma}

\begin{proposition}\label{prop:BoundaryIsLCH}
For $n \in \mathbb{N}$, the sets of the form $Z_n(U) \setminus Z_n(K)$ (resp. $Z(U) \setminus Z(K)$) for open sets $U \subseteq E^*$ and compact sets $K \subseteq E^*$ form a basis for a locally compact Hausdorff topology on $\partial E_n$ (resp. $\partial E$).

Moreover, for any compact set $K \subseteq E^*$, $Z_n(K)$ is compact in $\partial E_n$ and $Z(K)$ is compact in $\partial E$.
\end{proposition}

\begin{proof}
For notational convenience, we write $\partial E_\infty := \partial E$ and $Z_\infty(S) := Z(S)$ for $S \subseteq E^*$.  Fix $n \in \mathbb{N} \cup \{\infty\}$.

First we show the sets described above form a basis for $\partial E_n$.  Suppose $U, V \subseteq E^*$ are open and $K, L \subseteq E^*$ are compact and assume $\alpha \in (Z_n(U) \setminus Z_n(K)) \cup (Z_n(V) \setminus Z_n(L))$.  Then there are $k, \ell \in \mathbb{N}$ such that $k, \ell \leq |\alpha|$, $\alpha_1 \cdots \alpha_k \in U$ and $\alpha_1 \cdots \alpha_\ell \in V$ (if $k = 0$ or $\ell = 0$, these expressions are interpreted as $r(\alpha)$).  Without loss of generality, assume $k \leq \ell$.  Note that the set
\[ W := \{ \beta \in V : |\beta| \geq k \text{ and } \beta_1 \cdots \beta_k \in U \} \]
is open in $E^*$.  Also, $\alpha \in Z_n(W)$ and $Z_n(W) \subseteq Z_n(U) \cap Z_n(V)$.  Moreover, $K \cup L$ is compact and $Z_n(K) \cup Z_n(L) = Z_n(K \cup L)$.  As
\[ \alpha \in Z_n(W) \setminus Z_n(K \cup L) \subseteq ( Z_n(U) \setminus Z_n(K) ) \cap ( Z_n(V) \setminus Z_n(L) ), \]
the sets described in the statement of the proposition form a basis for $\partial E_n$.

To see the topology is Hausdorff, fix distinct paths $\alpha, \beta \in \partial E_n$.  Let $k = |\alpha|$ and $\ell = |\beta|$, and assume, without loss of generality, $k \leq \ell$.  If $k = \ell = \infty$, there is an $m \in \mathbb{N}$ such that $\alpha_1 \cdots \alpha_m \neq \beta_1 \cdots \beta_m$.  Fix disjoint open neighborhoods $U, V \subseteq E^m$ of $\alpha_1 \cdots \alpha_m$ and $\beta_1 \cdots \beta_m$, respectively.  Then $\alpha \in Z_n(U)$, $\beta \in Z_n(V)$, and $Z_n(U) \cap Z_n(V) = \emptyset$.  If $k < \infty$ and $\alpha_1 \cdots \alpha_k \neq \beta_1 \cdots \beta_k$, we can proceed as above.  The case that remains is when $k < \ell$ and $\alpha_1 \cdots \alpha_k = \beta_1 \cdots \beta_k$.  Fix a relatively compact open set $U \subseteq E^{k+1}$ containing $\beta_1 \cdots \beta_{k+1}$.  Then $Z_n(U)$ and $\partial E_n \setminus Z_n(\bar{U})$ are disjoint open sets containing $\beta$ and $\alpha$, respectively.  Hence $\partial E_n$ is Hausdorff.

Next we show $Z_n(K)$ is compact for every compact set $K \subseteq E^*$.  We may assume $K \subseteq E^k$ for some $k \in \mathbb{N}$ with $k \leq n$.  Let $(\alpha^i)$ be a net in $Z_n(K)$.  We first show there is an $\alpha \in \partial E_n$ such that for every that for every $\ell \in \mathbb{N}$ with $\ell \leq |\alpha|$, there is a subnet of $(\alpha_1^i \cdots \alpha_\ell^i)$ converging to $\alpha_1 \cdots \alpha_\ell$.  Note in particular that the choice of subnet may depend on $\ell$, but $\alpha$ does not.

Consider the net $(\alpha_1^i \cdots \alpha_k^i) \subseteq K$.  Since $K$ is compact, there is a subnet converging to a path  $\alpha_1 \cdots \alpha_k \in K$. If $k =n $ or $s(\alpha_k) \in E^0_\sng$, then $\alpha = \alpha_1 \cdots \alpha_k \in \partial E_n$ and the claim holds.  Otherwise, there is a compact neighborhood $L \subseteq E^0_\reg$ of $s(\alpha_k)$.  If $(\alpha_1^j \cdots \alpha_k^j)$ is a subnet of $(\alpha_1^i \cdots \alpha_k^i)$ converging to $\alpha_1 \cdots \alpha_k$, then for sufficiently large $j$, $s(\alpha_k^j) \in L \subseteq E^0_\reg$.  In particular, $\alpha_1^j \cdots \alpha_k^j \notin \partial E_n$ and hence $|\alpha^j| > k$.  Now, for large $j$, $\alpha^j_{k+1} \in r^{-1}(L)$.  Since $r^{-1}(L)$ is compact, there is a subnet of $\alpha^j_{k+1}$ converging to some $\alpha_{k+1}$.  Then $s(\alpha_{k+1}) = r(\alpha_k)$ by the continuity of $r$ and $s$ and hence $\alpha_1 \cdots \alpha_{k+1} \in E^{k+1}$.  Also, $(\alpha_1^i \cdots \alpha_k^i)$ has a subnet converging to $\alpha_1 \cdots \alpha_k$ by construction.  Continuing by induction yields the desired path $\alpha \in \partial E_n$.

Having constructed $\alpha$, we now show the net $(\alpha^i)$ has a subnet converging to $\alpha$.  Fix an open set $U \subseteq E^*$ and a compact set $L \subseteq E^*$ with $\alpha \in Z_n(U) \setminus Z_n(L)$.  We may assume $U \cap E^p$ and $L \cap E^p$ are empty for $p > n$.  Then for some $\ell \in \mathbb{N}$, $\alpha_1 \cdots \alpha_\ell \in U$.  As $L$ is compact, there is an integer $m$ with $\ell \leq m \leq n$ such that $E^p \cap L = \emptyset$ for $p > m$.  By the construction of $\alpha$, there is a subnet $(\alpha^j)$ of $(\alpha^i)$ such that $\alpha_1^j \cdots \alpha_m^j$ converges to $\alpha_1 \ldots \alpha_m$ in $E^m$.  Since $\ell \leq m$ and $U$ is open, for sufficiently large $j$, $\alpha_1^j \ldots \alpha_\ell^j \in U$.  Moreover, for each $0 \leq p \leq m$ and sufficiently large $j$, $\alpha_1^j \ldots \alpha_p^j \notin L$ since $L$ is closed.  In particular, $\alpha^j \in Z_n(U) \setminus Z_n(L)$ for large $j$ and hence
\[ \{ i \in I : \alpha^i \in Z_n(U) \setminus Z_n(L) \} \]
is cofinal in $I$.  By Lemma \ref{lem:subnet}, the net $(\alpha^i)$ has a subnet converging to $\alpha$.  Hence $Z_n(K)$ is compact.

In only remains to show $\partial E_n$ is locally compact.  Given any $\alpha \in \partial E_n$, choose a relatively compact open set $U \subseteq E^0$ containing $r(\alpha)$.  Then $Z_n(U)$ is an open set containing $\alpha$.  By the work above, $Z_n(\bar{U})$ is compact and $\alpha \in Z_n(U) \subseteq Z_n(\bar{U})$. This completes the proof.
\end{proof}

For each $n \geq 1$, define maps $\rho_n : \partial E_n \rightarrow \partial E_{n-1}$ by
\[ \rho_n(\alpha_1 \cdots \alpha_k) = \begin{cases} \alpha_1 \cdots \alpha_{n-1} & k = n \\ \alpha_1 \cdots \alpha_k & k < n \end{cases} \]
and define $\rho_{n, \infty} : \partial E \rightarrow \partial E_n$ similarly.

\begin{proposition}\label{prop:BoundaryProjectiveLimit}
The maps $\rho_n$ and $\rho_{n, \infty}$ are proper, continuous, and surjective.  Moreover, the maps $\rho_{\infty, n}$ induce a homeomorphism
\[ \partial E \longrightarrow \underset{\longleftarrow}{\lim} \, (\partial E_n, \rho_n). \]
\end{proposition}

\begin{proof}
It is easy to see the maps $\rho_n$ are surjective.  Indeed, if $\alpha \in \partial E_{n-1}$ with $s(\alpha) \in E^0_\sng$, then $\alpha$ is in the range of $\rho_n$.  Otherwise, $|\alpha| = n - 1$ and $s(\alpha) \in E^0_\reg$.  By the definition of $E^0_\reg$, there is an $e \in E^1$ with $r(e) = s(\alpha)$.  Hence $\alpha e \in E^n \subseteq \partial E_n$ and $\rho_n(\alpha e) = \alpha$.  Applying this argument by induction shows $\rho_{n, \infty}$ is surjective.

To see continuity, note that if $S \subseteq E^*$ and $S \cap E^n = \emptyset$, then $\rho_n^{-1}(Z_{n-1}(S)) = Z_n(S)$.  In particular, if $U \subseteq E^*$ is open and $K \subseteq E^*$ is compact with $U \cap E^n = K \cap E^n = \emptyset$, then $\rho_n^{-1}(Z_{n-1}(U) \setminus Z_{n-1}(K)) = Z_n(U) \setminus Z_n(K)$.  So $\rho_n$ is continuous.  If $K \subseteq \partial E_{n-1}$ is compact, there are relatively compact open sets $U_1, \ldots, U_m \subseteq E^*$ with $K \subseteq \bigcup_{i=1}^n Z_{n-1}(U_i)$.  Now,
\[ \rho_n^{-1}(K) \subseteq \bigcup_{i=1}^n \rho_n^{-1}(Z_{n-1}(U_i)) \subseteq \bigcup_{i=1}^n Z_n(\bar{U}_i). \]
As each $Z(\bar{U}_i)$ is compact by Proposition \ref{prop:BoundaryIsLCH} and $\rho_n^{-1}(K)$ is closed by the continuity of $\rho_n$, $\rho_n^{-1}(K)$ is compact.  Hence $\rho_n$ is proper.  Nearly identical arguments show $\rho_{n, \infty}$ is continuous and proper for each $n$.

It remains to show $\partial E$ is the projective limit of the $\partial E_n$.  Suppose $X$ is a topological space and $f_n : X \rightarrow \partial E_n$ are continuous maps such that $\rho_{n+1} f_{n+1} = f_n$ for each $n \geq 0$.  There is a unique function $f : X \rightarrow \partial E$ such that $\rho_{n, \infty}f = f_n$ for each $n \in \mathbb{N}$.  It suffices to verify $f$ is continuous.  Fix a relatively compact open set $U \subseteq E^*$ and a compact set $K \subseteq E^*$.  There is an $n \in \mathbb{N}$ such that $U \cap E^m$ and $K \cap E^m$ are empty for each $m > n$.  Then $Z(U) \setminus Z(K) = \rho_{n, \infty}^{-1}(Z_n(U) \setminus Z_n(K))$.  Now,
\[ f^{-1}(Z(U) \setminus Z(K)) = f^{-1}(\rho_{n, \infty}^{-1}(Z_n(U) \setminus Z_n(K))) = f_n^{-1}(Z_n(U) \setminus Z_n(K)) \]
is open by the continuity of $f_n$.  So $f$ is continuous.
\end{proof}

\begin{definition}\label{defn:BackwardsShift}
For $n \geq 1$, define the \emph{backwards shift} $\sigma_n : \partial E_n \setminus E^0 \rightarrow \partial E_{n-1}$ by
\[ \sigma(\alpha_1 \cdots \alpha_k) = \begin{cases} s(\alpha_1) & k = 1 \\ \alpha_2 \cdots \alpha_k &  k \geq 2. \end{cases} \]
and define $\sigma : \partial E \setminus E^0 \rightarrow \partial E$ similarly.
\end{definition}

Note that the maps $\sigma_n$ and $\rho_n$ satisfy the following compatibility conditions:
\[ \sigma_{n-1} \rho_n = \rho_{n-1} \sigma_n \qquad \text{ and } \qquad \sigma_n \rho_{n, \infty} = \rho_{n, \infty} \sigma. \]
on $\partial E_n \setminus E^0$ and $\partial E \setminus E^0$, respectively.
\begin{proposition}\label{prop:BackwardsShift}
The maps $\sigma_n : \partial E_n \setminus E^0 \rightarrow \partial E_{n-1}$ and $\sigma : \partial E \setminus E^0 \rightarrow \partial E$ are local homeomorphisms.
\end{proposition}

\begin{proof}
As in the proof of Proposition \ref{prop:BoundaryIsLCH}, we let $\partial E_\infty := \partial E$, $\sigma_\infty := \sigma$, and $Z_\infty(S) = Z(S)$ for all $S \subseteq E^*$.

First we claim the maps $\sigma'_1  = s: E^1 \rightarrow E^{0}$ and $\sigma_n' : E^n \rightarrow E^{n-1}$, $n \geq 2$, given by
\[ \sigma'_n(\alpha_1 \cdots \alpha_n) = \alpha_2 \cdots \alpha_n \]
are continuous and open.  Continuity is clear and the map $\sigma'_1 = s$ is open by the definition of a topological graph.  Fix $n \geq 2$ and let $U \subseteq E^n$ be open.  Fix $\alpha \in U$ and consider a net $(\beta^i) \subseteq E^{n-1}$ with $\beta^i \rightarrow \sigma'_n(\alpha)$.  It's enough to show $\beta^i \in \sigma'_n(U)$ for large $i$.  Let $V \subseteq E^1$ be an open neighborhood of $\alpha_1$ such that $s|V$ is a homeomorphism.  Then
\[ r(\beta^i) \rightarrow r(\sigma'_n(\alpha)) = s(\alpha_1) \in s(V). \]
Since $s$ is a local homeomorphism, $s(V)$ is open and hence $r(\beta^i) \in s(V)$ for large $i$.  Hence, for large $i$, there is an $e^i \in V$ with $s(e^i) = r(\beta^i)$.  As $s|V$ is a homeomorphism and $s(e^i) \rightarrow s(\alpha_1)$, we have $e^i \rightarrow \alpha_1$.  Hence $e^i\beta^i \rightarrow \alpha$.  So for sufficiently large $i$, $e^i\beta^i \in U$ and $\beta^i = \sigma_n'(e^i\beta^i) \in \sigma'_n(U)$.  This shows $\sigma'_n(U)$ is open.  Hence $\sigma'_n$ is open and continuous.

Let $\sigma' : E^* \setminus E^0 \rightarrow E^*$ denote the map induced by the $\sigma'_n$, $n \geq 1$.  Then $\sigma'$ is open and continuous.  Given an open set $U \subseteq E^*$ and a compact set $K \subseteq E^*$, and $n \in \mathbb{N} \cup \{\infty\}$, we have
\[ \sigma_n((Z_n(U) \setminus Z_n(K)) \setminus E^0) = Z_{n-1}(\sigma'_n(U \setminus E^0)) \setminus Z_{n-1}(\sigma'_n(K \setminus E^0)). \]
As $\sigma'_n$ is open and continuous, $\sigma'_n(U \setminus E^0) \subseteq E^*$ is open and $\sigma'_n(K \setminus E^0) \subseteq E^*$ is compact.  Therefore, $\sigma_n$ is open.

To show continuity, suppose $\alpha^i \in \partial E \setminus E^0$ is a net converging to $\alpha \in \partial E \setminus E^0$.  Let $U \subseteq E^*$ be an open set and $K \subseteq E^*$ be a compact set with $\sigma_n(\alpha) \in Z_{n-1}(U) \setminus Z_{n-1}(K)$.  For some $m \geq 1$, $|\alpha| \geq m$ and $\alpha_2 \cdots \alpha_m \in U$ (where this expression is interpreted as $s(\alpha_1)$ when $m = 1$).  It follows that $|\alpha^i| \geq m$ for all large $i$ and $\alpha_2^i \cdots \alpha_m^i \rightarrow \alpha_2 \cdots \alpha_m$ in $E^{m-1}$.  Hence $\sigma_n(\alpha^i) \in Z_{n-1}(U)$ for sufficiently large $i$.  A similar argument shows $\sigma_n(\alpha^i) \notin Z_{n-1}(K)$ for sufficiently large $i$ and hence $\sigma_n(\alpha^i) \rightarrow \sigma_n(\alpha)$ and $\sigma$ is continuous.

Having shown $\sigma_n$ is open and continuous, it only remains to show $\sigma$ is locally injective.  Fix $\alpha \in \partial E_n \setminus E^0$.  Let $V \subseteq E^1$ be an open neighborhood of $\alpha_1$ such that $s|V$ is injective.  Then $Z_n(V)$ is an open neighborhood of $\alpha \in \partial E_n \setminus E^0$.  Given $\beta, \beta' \in Z(V)$ with $\sigma_n(\beta) = \sigma_n(\beta')$, we have $\beta_1, \beta_1' \in V$ and
$s(\beta_1) = r(\sigma_n(\beta)) = r(\sigma_n(\beta')) = s(\beta'_1)$.  So $\beta_1 = \beta_1'$.  Now,
$\beta = \beta_1 \sigma_n(\beta) = \beta_1'\sigma_n(\beta') = \beta'$.
\end{proof}

\section{Invariant Measures on Topological Graphs}

Our goal is to describe the tracial states on a topological graph algebra $\cstar(E)$.  Note that any tracial state on $\cstar(E)$ yields a Radon probability measure on $E^0$ by composing the tracial state with the canonical map $\pi^0 : C_0(E^0) \rightarrow \cstar(E)$ and applying the Reisz Representation Theorem.  Definition \ref{defn:InvariantMeasure} below describes all measures on $E^0$ which arise in this way (see Theorem \ref{thm:InvariantTrace}).  We first introduce some notation.

Suppose $X$ and $Y$ are locally compact Hausdorff spaces and $\varphi : X \rightarrow Y$ is a local homeomorphism.  Given $f \in C_c(X)$, the function
\[ Y \rightarrow \mathbb{C} \, : \, y \mapsto \sum_{x \in \varphi^{-1}(y)} f(x) \]
is continuous and has compact support by Lemma \ref{lem:LocalHomeomorphism}.  Given a Radon measure $\mu$ on $Y$, let $\varphi^*\mu$ denote the unique Radon measure on $X$ satisfying
\[ \int_X f \, d\varphi^*\mu = \int_Y \sum_{x \in \varphi^{-1}(y)} f(x) \, d\mu(y) \]
for all $f \in C_c(X)$.

In particular, note that if $E$ is a topological graph, then $s : E^1 \rightarrow E^0$ is a local homeomorphism.  Thus a Radon measure $\mu$ on $E^0$ induces a Radon measure $s^*\mu$ on $E^1$.

\begin{definition}\label{defn:InvariantMeasure}
An \emph{invariant measure} $\mu$ on $E$ is a Radon probability measure $\mu$ on $E^0$ such that for all positive functions $f \in C_c(E^0)$,
\begin{equation}\label{eqn:InvariantMeasure}
   \int_{E^{^1}} f \circ r \, ds^*\mu \leq \int_{E^{^0}} f \, d \mu,
\end{equation}
with equality when $\operatorname{supp}(f) \subseteq E^0_\reg$.  Let $T(E)$ denote the set of invariant measures on $E$.
\end{definition}

\begin{remark}
If $\mu$ is an invariant measure on $E$, then $s^*\mu$ has mass at most 1.  To see this, note that \eqref{eqn:InvariantMeasure} holds for all positive $f \in C_b(E^0)$ by the Monotone Convergence Theorem.  Then taking $f = 1$ in \eqref{eqn:InvariantMeasure} shows $(s^*\mu)(E^1) \leq \mu(E^0) = 1$.  The measure $s^*\mu$ will not have mass 1 in general.  It is easy to construct examples of this even with finite graphs.
\end{remark}

We end this section by showing invariant measures on a topological graph $E$ induce measures on the boundary path space $\partial E$ satisfying a natural invariance condition.  First we need a measure theoretic lemma.

\begin{lemma}
Let $\mu$ be an invariant measure on $E$.  For any function $f \in C_b(E^0_\reg)$,
\begin{equation}\label{eqn:InvariantMeasureRegular}
  \int_{r^{^{-1}}(E^{^0}_\reg)} f \circ r \, ds^*\mu = \int_{E^{^0}_\reg} f \, d\mu,
\end{equation}
and for any positive function $f \in C_c(E^0_\sng)$,
\begin{equation}\label{eqn:InvariantMeasureSingular}
  \int_{r^{^{-1}}(E^{^0}_\sng)} f \circ r \, ds^*\mu \leq \int_{E^{^0}_\sng} f \, d \mu.
\end{equation}
\end{lemma}

\begin{proof}
We first prove \eqref{eqn:InvariantMeasureRegular}.  For any positive function $f \in C_c(E^0_\reg)$, this is an immediate consequence of Definition \ref{defn:InvariantMeasure}.  The result follows by the Monotone Convergence Theorem and the linearity of the integral using that $E^0_\reg$ is an open subset of $E^0$.

Now suppose $f \in C_c(E^0_\sng)$ is positive.  By the Tietze Extension Theorem, we may extend $f$ to a positive function, still denoted $f$, in $C_c(E^0)$.  As $f|E^0_\reg$ is a bounded, continuous function on $E^0_\reg$, \eqref{eqn:InvariantMeasureSingular} follows by subtracting \eqref{eqn:InvariantMeasureRegular} from \eqref{eqn:InvariantMeasure}.
\end{proof}

\begin{proposition}\label{prop:BoundaryMeasure}
Suppose $E$ is a topological graph and $\mu$ is an invariant measure on $E$.  There is a Radon probability measure $\tilde{\mu}$ on $\partial E$ such that \begin{align}
   \int_{\partial E} f \circ r \, d\tilde{\mu} &= \int_{E^{^0}} f \, d\mu  &\text{for each $f \in C_c(E^0)$} \label{eqn:BoundaryMeasureRange}
\intertext{and}
   \int_{\partial E \setminus E^{^0}} f \, d\sigma^*\tilde{\mu} &= \int_{\partial E} f \, d\tilde{\mu} &\text{for each $f \in C_c(\partial E \setminus E^0)$.} \label{eqn:BoundaryMeasureInvariant}
\end{align}
\end{proposition}

\begin{proof}
For $n \in \mathbb{N}$, we define measure $\tilde{\mu}_n$ on $\partial E_n$ inductively as follows.  Set $\tilde{\mu}_0 = \mu$ and given $\tilde{\mu}_n$, let $\tilde{\mu}_{n+1}$ denote the Radon measure on $\partial E_{n+1}$ determined by
\begin{equation}\label{eqn:BoundaryMeasureApproximations}
\int_{\partial E_{n+1}} f \, d\tilde{\mu}_{n+1} = \int_{\partial{E}_{n+1} \setminus E^{^0}} f \, d \sigma_{n+1}^* \tilde{\mu}_n + \int_{E^{^0}_\sng} f \, d\mu - \int_{r^{^{-1}}(E^{^0}_\sng)} f \circ r \, d s^*\mu
\end{equation}
for every $f \in C_c(\partial E_{n+1})$.  Note in particular that if $f \in C_c(\partial E_{n+1})$, then $f|E^0_\sng \in C_c(E^0_\sng)$ since $E^0_\sng$ is closed in $\partial E_{n+1}$ and hence $\tilde{\mu}_{n+1}$ is positive by \eqref{eqn:InvariantMeasureSingular}.

Next we show
\begin{equation}\label{eqn:ProjectionMapsMeasurePreserving}
  \int_{\partial E_{n+1}} f \circ \rho_{n+1} \, d\tilde{\mu}_{n+1} = \int_{\partial E_n} f \, d\tilde{\mu}_n
\end{equation}
for every $f \in C_c(\partial E_n)$.  Assume first $n = 0$ and $f \in C_c(\partial E_0)$.  Then since $\rho_1(v) = v$ for all $v \in E^0_\sng$,
\begin{align*}
     \int_{\partial E_1} f \circ \rho_1 \, d\tilde{\mu}_1
  &\overset{\eqref{eqn:BoundaryMeasureApproximations}}{=} \int_{E^{^1}} f \circ r \, ds^*\mu + \int_{E^{^0}_\sng} f \, d\mu - \int_{r^{^{-1}}(E^{^0}_\sng)} f \circ r \, ds^*\mu \\
  &= \int_{r^{^-1}(E^{^0}_\reg)} f \circ r \, ds^*\mu + \int_{E^{^0}_\sng} f \, d\mu \\
  &\overset{\eqref{eqn:InvariantMeasureRegular}} = \int_{E^{^0}_\reg} f \, d \mu + \int_{E^{^0}_\sng} f\, d\mu = \int_{E^{^0}} f \, d\mu.
\end{align*}
So \eqref{eqn:ProjectionMapsMeasurePreserving} holds when $n = 0$.

Suppose $n \geq 1$ and \eqref{eqn:ProjectionMapsMeasurePreserving} holds for $n - 1$.  Let $f \in C_c(\partial E_n \setminus E^0)$ be given.  Since $\sigma_n : \partial E_n \setminus E^0 \rightarrow \partial E_{n-1}$ is a local homeomorphism, the function $g : \partial E_{n-1} \rightarrow \mathbb{C}$ defined by
\[ g(\alpha) = \sum_{\sigma_n(\beta) = \alpha} f(\beta) = \sum_{s(e) = r(\alpha)} f(e\alpha) \]
is continuous and has compact support by Lemma \ref{lem:LocalHomeomorphism}.  Now we compute
\begin{align*}
     \int_{\partial E_{n+1}} f \circ \rho_{n+1} \, d \sigma_{n+1}^* \tilde{\mu}_n
  &= \int_{\partial E_n} \sum_{s(e) = r(\alpha)} f(\rho_{n+1}(e\alpha)) \, d\tilde{\mu}_n(\alpha) \\
  &= \int_{\partial E_n} \sum_{s(e) = r(\rho_n(\alpha))} f(e\rho_n(\alpha)) \, \tilde{\mu}_n(\alpha) \\
  &= \int_{\partial E_n} g \circ \rho_n \, d \tilde{\mu}_n \overset{\eqref{eqn:ProjectionMapsMeasurePreserving}}{=} \int_{\partial E_{n-1}} g \, d\tilde{\mu}_{n-1} \\
  &= \int_{\partial E_{n-1}} \sum_{s(e) = r(\alpha)} f(e\alpha) \, d\tilde{\mu}_{n-1} = \int_{\partial E_n \setminus E^{^0}} f \, d \sigma_n^* \tilde{\mu}_{n-1}.
\end{align*}
This shows
\begin{equation}\label{eqn:ProjectionMapsMeasurePreserving'}
  \int_{\partial E_{n+1} \setminus E^{^0}} f \circ \rho_{n+1} \, d \sigma_{n+1}^* \tilde{\mu}_{n+1} = \int_{\partial E_n \setminus E^{^0}} f \, d \sigma_n^* \tilde{\mu}_{n-1}
\end{equation}
for all $f \in C_c(\partial E_n \setminus E^0)$ and hence for all $f \in C_b(\partial E_n \setminus E^0)$ by the Monotone Converge Theorem using that $\partial E_n \setminus E^0$ is an open subset of $\partial E^n$.

Now, assuming $n \geq 1$ and \eqref{eqn:ProjectionMapsMeasurePreserving} holds for $n - 1$, let $f \in C_c(\partial E_n)$.  Using \eqref{eqn:ProjectionMapsMeasurePreserving'} and the fact that $\rho_{n+1}$ restricts to the identity map on $E^0_\sng$, we have
\begin{align*}
     \int_{\partial E_{n+1}} &f \circ \rho_{n+1} \, d\tilde{\mu}_{n+1} \\
  &\overset{\eqref{eqn:BoundaryMeasureApproximations}}{=} \int_{\partial{E_{n+1} \setminus E^{^0}}} f \circ \rho_{n+1} \, \sigma_{n+1}^* \tilde{\mu}_n + \int_{E^{^0}_\sng} f \circ \rho_{n+1} \, d\mu - \int_{r^{^{-1}}(E^{^0}_\sng)} f \circ \rho_{n+1} \circ r \, d \mu \\
  &\overset{\eqref{eqn:ProjectionMapsMeasurePreserving'}}{=} \int_{\partial{E_n \setminus E^{^0}}} f \, d \sigma_n^* \tilde{\mu}_{n-1} + \int_{E^{^0}_\sng} f \, d\mu - \int_{r^{^{-1}}(E^{^0}_\sng)} f \circ r \, d\mu \\
  &\overset{\eqref{eqn:BoundaryMeasureApproximations}}{=} \int_{\partial E_n} f \, d\tilde{\mu}_n,
\end{align*}
which verifies \eqref{eqn:ProjectionMapsMeasurePreserving}.

Note that \eqref{eqn:ProjectionMapsMeasurePreserving} holds for all continuous bounded functions $f$ on $\partial E_{n+1}$, and in particular, taking $f = 1$ in \eqref{eqn:ProjectionMapsMeasurePreserving} shows $\tilde{\mu}_n$ is a probability measure for each $n \geq 0$ since $\tilde{\mu}_0 = \mu$ is.  Now, using \eqref{eqn:ProjectionMapsMeasurePreserving}, Proposition \ref{prop:BoundaryProjectiveLimit}, and Gelfand Duality, the measures $\tilde{\mu}_n$ on $\partial E_n$ induce a Radon probability measure $\tilde{\mu}$ on $\partial E$ such that
\begin{equation}\label{eqn:BoundaryMeasureProjection}
  \int_{\partial E} f \circ \rho_{\infty, n} d\tilde{\mu} = \int_{\partial E_n} f d\tilde{\mu}_n
\end{equation}
for every $f \in C_c(\partial E_n)$ and $n \in \mathbb{N}$.

It is immediate from the construction that \eqref{eqn:BoundaryMeasureRange} holds.  To show \eqref{eqn:BoundaryMeasureInvariant}, it suffices to show
\begin{equation}\label{eqn:BoundaryMeasureApproxInvariance}
  \int_{\partial E \setminus E^{^0}} f \circ \rho_{n, \infty} \, d \sigma^* \tilde{\mu} = \int_{\partial E} f \circ \rho_{n, \infty} \, d \tilde{\mu}
\end{equation}
for each $n \geq 1$ and $f \in C_c(\partial E_n \setminus E^0)$.  Calculations similar to those above yield
\[ \int_{\partial E \setminus E^{^0}} f \circ \rho_{n, \infty} \, d \sigma^* \tilde{\mu} = \int_{\partial E_n \setminus E^{^0}} f \, d \sigma_n^* \tilde{\mu}_{n-1}. \]
By the definition of $\tilde{\mu}_n$, if $f \in C_c(\partial E \setminus E^0)$, then
\[ \int_{\partial E_n \setminus E^{^0}} f \, d \sigma_n^* \tilde{\mu}_{n-1} \overset{\eqref{eqn:BoundaryMeasureApproximations}}{=} \int_{\partial E} f \, d\tilde{\mu}_n \overset{\eqref{eqn:BoundaryMeasureProjection}}{=} \int_{\partial E} f \circ \rho_{n, \infty} \, d \tilde{\mu}. \]
The last two equalities verify \eqref{eqn:BoundaryMeasureApproxInvariance} and hence \eqref{eqn:BoundaryMeasureInvariant} holds.
\end{proof}

For a vertex $v \in E^0$, write $E^k v$ for the set of paths $\alpha \in E^k$ such that $s(\alpha) = v$.  Applying \eqref{eqn:BoundaryMeasureInvariant} by induction yields the following corollary.

\begin{corollary}\label{cor:BoundaryMeasureInvariant}
Let $\mu$ be an invariant measure on $E$ and let $\tilde{\mu}$ denote the induced measure on $\partial E$ as in Proposition \ref{prop:BoundaryMeasure}.  For any integer $k \geq 1$ and $f \in C_c(\partial E \setminus \partial E_{k-1})$,
\begin{equation}\label{eqn:BoundaryMeasureInvariant'}
   \int_{\partial E}  \sum_{\beta \in E^k r(\alpha)} f(\beta\alpha) \, d\tilde{\mu}(\alpha) = \int_{\partial E} f \, d\tilde{\mu}
\end{equation}
\end{corollary}

\section{The Boundary Representation}

In this section, we construct a Hilbert $C_0(\partial E)$-module $L^2(E)$ and a faithful representation $\pi : \cstar(E) \rightarrow \mathbb{B}(L^2(E))$.  Roughly, this can be thought of as the left regular representation of the graph on it's boundary path space.  Moreover, if $\mu$ is an invariant measure on $E$, there is an induced representation $\pi_\mu$ of $\cstar(E)$ on a Hilbert space $L^2(E, \mu)$ (see Definition \ref{defn:GNSSpace}).  This representation will be used in the next section to show $\mu$ induces a tracial state on $\cstar(E)$.

The following is a standard construction in the groupoid literature (see, for example Definition 2.4 in \cite{Renault}).  Define the subset $X \subseteq \partial E \times \mathbb{Z} \times \partial E$ by
\[ X = \{ (\alpha, k - \ell, \beta) : k, \ell \in \mathbb{N}, \, \alpha, \beta \in \partial E, \, |\alpha| \geq k, \, |\beta| \geq \ell, \, \sigma^k(\alpha) = \sigma^\ell(\beta) \}. \]
Since $\sigma : \partial E \setminus E^0 \rightarrow \partial E$ is a local homeomorphism, the sets
\[ Z(U, k, \ell, V) := \{ (\alpha, k - \ell, \beta) \in X : \alpha \in U, \beta \in V, \sigma^k(\alpha) = \sigma^\ell(\beta) \} \]
where $k, \ell \in \mathbb{N}$, $U \subseteq \partial E \setminus \partial E_{k - 1}$ and $V \subseteq \partial E \setminus \partial E_{\ell - 1}$ are open, and $\sigma^k$ (resp. $\sigma^\ell$) is a homeomorphisms on $U$ (resp. $V$) form a basis for a locally compact Hausdorff topology on $X$.  Moreover, the map $d : X \rightarrow \partial E$ given by $d(\alpha, n, \beta) = \beta$ is a local homeomorphism.

\begin{remark}
After identifying $\partial E$ with the subspace $\{(\beta, 0, \beta) : \beta \in \partial E \} \subseteq X$, the space $X$ forms a locally compact, Hausdorff, amenable, \`{e}tale groupoid with unit space $\partial E$.  Moreover the \cstar-algebra generated by $X$ is precisely $\cstar(E)$ as was shown by Yeend in \cite{Yeend2} (see also \cite{KumjianLi} where the twisted analogue of this result is proven).
\end{remark}

We recall the from Definition 6.4 in \cite{KumjianLi} the characterization of convergent nets in $X$.

\begin{lemma}\label{lem:ConvergentNetsInX}
Consider a net $(\alpha^i, n^i, \beta^i) \subseteq X$ and a point $(\alpha, n, \beta) \in X$.  Find integers $k, \ell \geq 0$ such that
\begin{enumerate}
  \item $n = k - \ell$, $|\alpha| \geq k$, $|\beta| \geq \ell$, and $\sigma^k(\alpha) = \sigma^\ell(\beta)$, and
  \item If $k', \ell' \geq 0$ satisfy (1) with $k' \leq k$ and $\ell' \leq \ell$, then $k = k'$ and $\ell = \ell'$.
\end{enumerate}
Then $(\alpha^i, n^i, \beta^i)$ converges to $(\alpha, n, \beta)$ if and only if $\alpha^i \rightarrow \alpha$, $\beta^i \rightarrow \beta$, and for all sufficiently large $i$, we have $n^i = n$, $|\alpha^i| \geq k$, $|\beta^i| \geq \ell$, and $\sigma^k(\alpha^i) = \sigma^\ell(\beta^i)$.
\end{lemma}

\begin{definition}
The space $C_c(X)$ is a right $C_0(\partial E)$-module with the action
\[ (f \cdot b)(\alpha, n, \beta) = f(\alpha, n, \beta) b(\beta) \]
for $f \in C_c(X)$, $b \in C_0(\partial E)$, and $(\alpha, n, \beta) \in X$.  Define a $C_0(\partial E)$-valued inner product on $C_c(X)$ by
\[ \langle f, g \rangle(\beta) = \sum_{d(x) = \beta} \overline{f(x)} g(x) \]
for each $f, g \in C_c(X)$ and $\beta \in \partial E$.  Let $L^2(E)$ denote the Hilbert $C_0(\partial E)$-module obtained by completing $C_c(X)$ with respect to the inner product above.
\end{definition}

We will construct a covariant Toeplitz representation of $E$ on $L^2(E)$.  First we isolate a few technical lemmas which will be used throughout the section.  Define
\[ X_0 := \{ (\alpha, n, \beta) \in X : |\alpha| \geq 1 \} \]
and let $d_0 : X_0 \rightarrow \partial E$ denote the restriction of $d$; that is, $d_0(\alpha, n, \beta) = \beta$.

\begin{lemma}
For any $\beta \in \partial E$,
\[ d_0^{-1}(\beta) = \{ (e\alpha, n + 1, \beta) : (\alpha, n, \beta) \in X, \, e \in E^1, \text{ and } s(e) = r(\alpha) \} \]
In particular, given a funtion $f : X_0 \rightarrow \mathbb{C}$ and $\beta \in \partial E$,
\begin{equation}\label{eqn:Reindexing}
  \sum_{x \in d_0^{-1}(\beta)} f(x) = \sum_{(\alpha, n, \beta) \in X} \sum_{s(e) = r(\alpha)} f(e\alpha, n + 1, \beta)
\end{equation}
whenever one (and hence both) sums exist.
\end{lemma}

\begin{lemma}\label{lem:ShiftLocalHomeomorphism}
The map $\tilde{\sigma} : X_0 \rightarrow X$ given by
\[ \tilde{\sigma}(\alpha, n, \beta) = (\sigma(\alpha), n - 1, \beta) \]
is a local homeomorphism.
\end{lemma}

\begin{proof}
Given a net $(\alpha^i, n^i, \beta^i) \in X_0$ converging to $(\alpha, n, \beta) \in X_0$, it is easy to show $(\sigma(\alpha^i), n - 1, \beta^i)$ converges to $(\sigma(\alpha), n - 1, \beta)$ using Lemma \ref{lem:ConvergentNetsInX}.  Hence $\tilde{\sigma}$ is continuous.  Given a basic open set $Z(U, k, \ell, V)$ in $X_0$ as in the definition of the topology on $X$, we have
\[ \tilde{\sigma}(Z((U), k, \ell, V) \cap X_0) = Z(\sigma(U \setminus E^0), k, \ell, V). \]
As $\sigma$ is open, this shows $\tilde{\sigma}$ is open.  That $\tilde{\sigma}$ is locally injective follows follows from the fact that $\sigma$ is locally injective.
\end{proof}

\begin{lemma}\label{lem:ContinuityPi1}
Given $f \in C_c(X)$ and $\xi \in C_c(E^1)$, the function $g : X \rightarrow \mathbb{C}$ given by
\[ g(\alpha, n, \beta) = \begin{cases} \xi(\alpha_1) f(\sigma(\alpha), n - 1, \beta) & |\alpha| \geq 1 \\ 0 & |\alpha| = 0 \end{cases} \]
is continuous and $\operatorname{supp}(g) \subseteq X_0$ is compact.
\end{lemma}

\begin{proof}
First we show $g: X \rightarrow \mathbb{C}$ is continuous.  Suppose $(\alpha^i, n^i, \beta^i) \rightarrow (\alpha, n, \beta)$ in $X$.  Then, in particular, $\alpha^i \rightarrow \alpha$ in $\partial E$.  Assuming $|\alpha| \geq 1$, we have $|\alpha^i| \geq 1$ for sufficiently large $i$ since $\partial E \setminus E^0$ is open in $\partial E$.  By the continuity of $\tilde{\sigma}$ (Lemma \ref{lem:ShiftLocalHomeomorphism}), $(\sigma(\alpha^i), n^i - 1, \beta^i) \rightarrow (\sigma(\alpha), n - 1, \beta)$.  Also, $\alpha^i \rightarrow \alpha$ and by the continuity of $\rho_1$ (Lemma \ref{prop:BoundaryProjectiveLimit}), $\alpha^i_1 \rightarrow \alpha_1$.  Hence by the continuity of $\xi$ and $f$, we have
\[ g(\alpha^i, n^i, \beta^i) \rightarrow g(\alpha, n, \beta). \]
Now suppose $|\alpha| = 0$.  Since $\operatorname{supp}(\xi) \subseteq E^1$ is compact, the set $U := \partial E \setminus Z(\operatorname{supp}(\xi))$ is open.  Since $\alpha \in U$, $\alpha^i \in U$ for all sufficiently large $i$.  Thus for large $i$, either $|\alpha^i| = 0$ or $|\alpha^i| \geq 1$ and $\alpha^i_1 \notin \operatorname{supp}(\xi)$.  In either case, for large $i$,
\[ g(\alpha^i, n^i, \beta^i) = g(\alpha, n, \beta) = 0. \]
This proves the continuity of $g$.  Note that this also shows $\operatorname{supp}(g) \subseteq X_1$.

To show $g$ has compact support, suppose $((\alpha^i, n^i, \beta^i))_{i \in I}$ is a net in the support of $g$.  Then for each $i$, $|\alpha^i| \geq 1$, $\alpha_1^i \in \operatorname{supp}(\xi)$, and $(\sigma(\alpha^i), n^i - 1, \beta^i) \in \operatorname{supp}(f)$.  Hence there is a subnet $((\alpha^{i(j)}, n^{i(j)}, \beta^{i(j)}))_{j \in J}$ and points $e \in \operatorname{supp}(\xi)$ and $(\alpha', n, \beta) \in \operatorname{supp}(f)$ such that $\alpha_1^{i(j)} \rightarrow e$ and $(\sigma(\alpha^{i(j)}), n^{i(j)} - 1, \beta^{i(j)}) \rightarrow (\alpha', n, \beta)$.  Using the characterization of convergent nets in $X$ given in Lemma \ref{lem:ConvergentNetsInX}, it is easy to show $(\alpha^{i(j)}, n^{i(j)}, \beta^{i(j)})$ converges to $(e \alpha', n + 1, \beta) \in X$.  Hence $g$ has compact support.
\end{proof}

We now construct the desired representation of $\cstar(E)$.  Recall from Definition \ref{defn:GraphCorrespondence}, $H(E)$ is the completion of $C_c(E^1)$ with respect to the $C_0(E^0)$-valued inner product induced by the local homeomorphism $s : E^1 \rightarrow E^0$.  It is easily verified that there is a $\ast$-homomorphism $\pi^0 : C_0(E^0) \rightarrow \mathbb{B}(L^2(E))$ such that
\[ \pi^0(a)(f)(\alpha, n, \beta) = a(r(\alpha)) f(\alpha, n, \beta) \]
for each $a \in C_0(E^0)$, $f \in C_c(X)$, and $(\alpha, n, \beta) \in X$.  The following proposition described the linear map $\pi^1 : H(E) \rightarrow \mathbb{B}(L^2(E))$.  The covariance is checked in Proposition \ref{prop:LeftRegularRepresentation}.  The representation is shown to be faithful and gauge invariant in Theorem \ref{thm:GaugeActionOnBL2}.

\begin{proposition}
There is a contractive, linear map $\pi^1 : H(E) \rightarrow \mathbb{B}(L^2(E))$ such that
\[ \pi^1(\xi)(f)(\alpha, n, \beta) = \begin{cases} \xi(\alpha_1) f(\sigma(\alpha), n - 1, \beta) & |\alpha| \geq 1 \\ 0 & |\alpha| = 0 \end{cases} \]
for each $\xi \in C_c(E^1)$, $f \in C_c(X)$, and $(\alpha, n, \beta) \in X$.  Moreover, we have
\[ \pi^1(\xi)^*(f)(\alpha, n, \beta) = \sum_{s(e) = r(\alpha)} \overline{\xi(e)} f(e\alpha, n + 1, \beta) \]
for each $\xi \in C_c(E^1)$, $f \in C_c(X)$, and $(\alpha, n, \beta) \in X$.
\end{proposition}

\begin{proof}
Lemma \ref{lem:ContinuityPi1} shows $\pi^1(\xi)(f) \in C_c(X)$.  It is clear that $\pi^1(\xi) : C_c(X) \rightarrow C_c(X)$ is linear.  To show $\pi^1(\xi)$ is bounded, let $f \in C_c(X)$ be given and note that
\begin{align*}
     \|\pi^1(\xi)(f)\|^2
  &= \sup_{\beta \in \partial E} \sum_{x \in d^{-1}(\beta)} |\pi^1(\xi)(f)(x)|^2 \\
  &= \sup_\beta \sum_{(\alpha, n, \beta), |\alpha| \geq 1} |\pi^1(\xi)(f)(\alpha, n, \beta)|^2 \\
  &\overset{\eqref{eqn:Reindexing}}{=} \sup_\beta \sum_{(\alpha, n, \beta)} \sum_{s(e) = r(\alpha)} |\pi^1(\xi)(f)(e\alpha, n + 1, \beta)|^2 \\
  &= \sup_\beta \sum_{(\alpha, n, \beta)} \sum_{s(e) = r(\alpha)} |\xi(e)|^2 |f(\alpha, n, \beta)|^2 \\
  &= \sup_\beta \sum_{(\alpha, n, \beta)} \langle \xi, \xi \rangle(r(\alpha)) |f(\alpha, n, \beta)|^2 \\
  &\leq \|\xi\|^2 \sup_\beta \sum_{(\alpha, n, \beta)} |f(\alpha, n, \beta)|^2 \\ 
  &= \|\xi\|^2 \|f\|^2
\end{align*}
Hence $\pi^1(\xi)$ is contractive and extends to a contractive linear map $L^2(E) \rightarrow L^2(E)$.

It remains to show $\pi^1(\xi)$ is adjointable with the adjoint given as above.  To this end, fix $\xi \in C_c(E^1)$ and define $T : C_c(X) \rightarrow C_c(X)$ by
\[ T(f)(\alpha, n, \beta) = \sum_{s(e) = r(\alpha)} \overline{\xi(e)} f(e \alpha, n + 1, \beta). \]
we claim $T(f) \in C_c(X)$.  Define $g : X_1 \rightarrow \mathbb{C}$ by $g(\alpha, n, \beta) = \overline{\xi(e)} f(\sigma(\alpha), n - 1, \beta)$ and note that $g \in C_c(X_1)$ by Lemma \ref{lem:ContinuityPi1}.  Now, for $x \in X$,
\[ T(f)(x) = \sum_{\tilde{\sigma}(y) = x} g(y) \]
and $T(f) \in C_c(X_1)$ by Lemma \ref{lem:LocalHomeomorphism} since $\tilde{\sigma} : X_1 \rightarrow X$ is a local homeomorphism by Lemma \ref{lem:ShiftLocalHomeomorphism}.

It is clear that $T(f)$ is linear.  We show $T(f)$ is bounded and hence extends to a linear map $L^2(E) \rightarrow L^2(E)$.  Fix $f \in C_c(X)$ and note
\begin{align*}
     \|T(f)\|^2
  &= \sup_{\beta \in \partial E} \sum_{d(x) = \beta} |T(f)(x)|^2 \\
  &= \sup_\beta \sum_{(\alpha, n, \beta)} \sum_{s(e) = r(\alpha)} |\overline{\xi(e)}|^2 |f(e\alpha, n + 1, \beta)|^2 \\
  &= \sup_\beta \sum_{(\alpha, n, \beta)} \sum_{s(e) = r(\alpha)} \langle \xi, \xi \rangle(r(\alpha)) |f(e\alpha, n + 1, \beta)|^2 \\
  &\overset{\eqref{eqn:Reindexing}}{=} \sup_\beta \sum_{(\alpha, n, \beta), |\alpha| \geq 1} \langle \xi, \xi \rangle(r(\alpha))|f(\alpha, n, \beta)|^2 \\
  &\leq \|\xi\|^2 \sup_\beta \sum_{(\alpha, n, \beta)} |f(\alpha, n, \beta)|^2 \\
  &= \|\xi\|^2 \|f\|^2
\end{align*}
So $T$ is contractive as claimed.

Finally, note that for each $f, g \in C_c(X)$ and $\beta \in \partial E$,
\begin{align*}
     \langle \pi^1(\xi)(f), g \rangle(\beta)
  &= \sum_{d(x) = \beta} \overline{\pi^1(\xi)(f)(x)} g(x) \\
  &= \sum_{(\alpha, n, \beta), |\alpha| \geq 1} \overline{\xi(\alpha_1) f(\sigma(\alpha), n - 1, \beta)} g(\alpha, n, \beta) \\
  &\overset{\eqref{eqn:Reindexing}}{=} \sum_{(\alpha, n, \beta)} \sum_{s(e) = r(\alpha)} \overline{\xi(e) f(\alpha, n , \beta)} g(e \alpha, n + 1, \beta) \\
  &= \sum_{(\alpha, n, \beta)} \overline{f(\alpha, n, \beta)} \sum_{s(e) = r(\alpha)} \overline{\xi(e)} g(e \alpha, n + 1, \beta) \\
  &= \sum_{(\alpha, n, \beta)} \overline{f(\alpha, n, \beta)} T(g)(\alpha, n , \beta) \\
  &= \langle f, Tg \rangle(\beta).
\end{align*}
So $\pi^1(\xi)$ is adjointable and the adjoint is as claimed in the statement of the proposition.

It is clear that $\pi^1 : C_c(E^1) \rightarrow \mathbb{B}(L^2(E))$ is linear and $\pi^1$ is contractive by the calculations above.  So $\pi^1$ extends to a linear contractive map on $H(E)$ completing the proof.
\end{proof}

\begin{proposition}\label{prop:LeftRegularRepresentation}
The pair $(\pi^0, \pi^1)$ constructed above is a covariant representation of $E$ and hence induces a $\ast$-homomorphisms $\pi : \cstar(E) \rightarrow \mathbb{B}(L^2(E))$.
\end{proposition}

\begin{proof}
First note that if $a \in C_c(E^0)$, $\xi \in C_c(E^1)$, and $f \in C_c(X)$, then
\begin{align*}
  \pi^0(a) \pi^1(\xi)(f)(\alpha, n, \beta) &= a(r(\alpha)) \pi^1(\xi)(\alpha, n, \beta) = a(r(\alpha_1)) \xi(\alpha_1) f(\sigma(\alpha), n - 1, \beta) \\
  &= (a\xi)(\alpha_1) f(\sigma(\alpha), n - 1, \beta) = \pi^1(a\xi)(f)(\alpha, n, \beta)
\end{align*}
for all $(\alpha, n, \beta) \in X$ with $|\alpha| \geq 1$.  Also, when $(\alpha, n, \beta) \in X$ with $|\alpha| = 0$,
\[ \pi^0(a)\pi^1(\xi)(f)(\alpha, n, \beta) = \pi^1(a\xi)(f)(\alpha,n,\beta) = 0. \]
Hence $\pi^0(a)\pi^1(\xi) = \pi^1(a\xi)$ for all $a \in C_0(E^0)$ and $\xi \in H(E)$.

Now, for $\xi, \eta \in C_c(E^0)$, $f \in C_c(E^1)$, and $(\alpha, n, \beta) \in X$, we have
\begin{align*}
   \pi^1(\xi)^*\pi^1(\eta)(f)(\alpha, n, \beta) &= \sum_{s(e) = r(\alpha)} \overline{\xi(e)} \pi^1(\eta)(f)(e \alpha, n + 1, \beta) = \sum_{s(e) = r(\alpha)} \overline{\xi(e)} \eta(e) f(\alpha, n, \beta) \\
   &= \langle \xi, \eta \rangle(r(\alpha)) f(\alpha, n, \beta) = \pi^0(\langle \xi, \eta \rangle)(f)(\alpha, n, \beta).
\end{align*}
So $\pi^1(\xi)^*\pi^1(\eta) = \pi^0(\langle \xi, \eta \rangle)$ for all $\xi, \eta \in H(E)$.

It remains to verify covariance.  The pair $(\pi^0, \pi^1)$ induces a map $\varphi: \mathbb{K}(H(E)) \rightarrow \mathbb{B}(L^2(E))$ such that $\varphi(\xi \otimes \eta^*) = \pi^1(\xi) \pi^1(\eta)^*$, where $\xi \otimes \eta^* : H(E) \rightarrow H(E)$ is the rank one operator given by $(\xi \otimes \eta^*)(\zeta) = \xi \langle \eta, \zeta \rangle$.  Let $\lambda : C_b(E^1) \rightarrow \mathbb{B}(L^2(E))$ denote the map given by pointwise multiplication; that is, $\lambda(\xi)(\eta)(x) = \xi(x) \eta(x)$.  We need to show $\varphi(\lambda(a \circ r)) = \pi^0(a)$ for each $a \in C_0(E^0_\reg)$.

Suppose $a \in C_c(E^0_\reg)$.  Then $a \circ r \in C_c(E^1)$ and by Lemma \ref{lem:PartitionOfUnity}, there are $n \in \mathbb{N}$ and $\xi_i, \eta_i \in C_c(E^1)$ such that $\lambda(a \circ r) = \sum_{i=1}^n \xi_i \otimes \eta_i^*$, $a \circ r = \sum_{i=1}^n \xi_i \overline{\eta_i}$, and if $e, e' \in E^1$ are distinct edges with $s(e) = s(e')$, then $\xi_i(e) \eta_i(e') = 0$.  For $f \in C_c(X)$ and $(\alpha, n, \beta) \in X$ with $|\alpha| \geq 1$,
\begin{align*}
     \varphi(\lambda(a \circ r))(f)(\alpha, n, \beta)
  &= \sum_{i=1}^m \pi^1(\xi) \pi^1(\eta)^*(f)(\alpha, n, \beta) \\
  &= \sum_{i=1}^m \xi(\alpha_1) \pi^1(\eta)^*(f)(\sigma(\alpha), n - 1, \beta) \\
  &= \sum_{i=1}^m \xi(\alpha_1) \sum_{s(e) = s(\alpha_1)} \overline{\eta(e)} f(e \sigma(\alpha), n, \beta) \\
  &= \sum_{i=1}^m \sum_{s(e) = s(\alpha_1)} \xi(\alpha_1) \overline{\eta(e)} f(e\sigma(\alpha), n, \beta) \\
  &= \sum_{i=1}^m \xi(\alpha_1) \overline{\eta(\alpha_1)} f(\alpha, n, \beta) \\
  &= a(r(\alpha_1)) f(\alpha, n, \beta) = \pi^0(a)(f)(\alpha, n, \beta).
\end{align*}
When $|\alpha| = 0$, $\varphi(\lambda(a \circ r))(f)(\alpha, n, \beta) = 0$.  Also, if $|\alpha| = 0$, then $\alpha = r(\alpha) \in E^0_\sng$ and hence $a(r(\alpha)) = 0$.  So $\pi^0(a)(f)(\alpha, n, \beta) = 0$.  Therefore, $\varphi(\lambda(a)) = \pi^0(a)$.

This shows the pair $(\pi^0, \pi^1)$ is covariant and hence induces a representation $\pi : \cstar(E) \rightarrow \mathbb{B}(L^2(E))$.
\end{proof}

We next show $\pi$ is faithful.  There is a natural $\mathbb{Z}$-grading on $L^2(E)$ which induces an action of $\mathbb{T}$ on $\mathbb{B}(L^2(E))$.  From here, the faithfulness of $\pi$ will follow from the Gauge Invariant Uniqueness Theorem.  We first describe the action on $\mathbb{B}(L^2(E))$.

\begin{definition}\label{defn:GaugeActionOnBL2}
Define unitaries $U_z : L^2(E) \rightarrow L^2(E)$ by
\[ U_z(f)(\alpha, n, \beta) = z^n f(\alpha, n, \beta) \]
for $f \in C_c(X)$ and $(\alpha, n, \beta) \in X$.  Now define a strongly continuous action $\gamma'$ of $\mathbb{T}$ on $\mathbb{B}(L^2(E))$ by $\gamma'_z(T) = U_z T U_z^*$ for $T \in \mathbb{B}(L^2(E))$ and $z \in \mathbb{T}$.  The action $\gamma'$ will be called the \emph{gauge action} on $\mathbb{B}(L^2(E))$.
\end{definition}

\begin{theorem}\label{thm:GaugeActionOnBL2}
The representation $\pi : \cstar(E) \rightarrow \mathbb{B}(L^2(E))$ constructed above is a faithful and preserves the gauge actions.
\end{theorem}

\begin{proof}
The gauge invariance is easily verified.  Now, by the Gauge Invariance Uniqueness Theorem, to show $\pi$ is faithful, it's enough to show $\pi^0$ is faithful.  Fix a non-zero $a \in C_0(E^0)$ and fix a vertex $v \in E^0$ with $a(v) \neq 0$.  There is an $\alpha \in \partial E$ such that $r(\alpha) = v$.  Let $f \in C_c(X)$ be such that $f(\alpha, 0, \alpha) \neq 0$.  Then
\[ \pi^0(a)(f)(\alpha, 0, \alpha) = a(r(\alpha)) f(\alpha, 0, \alpha) \neq 0. \]
In particular, $\pi^0(a) \neq 0$.  So $\pi^0$ is injective and hence $\pi$ is injective.
\end{proof}

For each invariant measure $\mu$ on $E$, we construct a Hilbert space $L^2(E, \mu)$ and a representation $\pi_\mu$ on $\cstar(E)$ on $L^2(E, \mu)$.  This representation is the key to lifting the measure $\mu$ to a tracial state on $\cstar(E)$.

\begin{definition}\label{defn:GNSSpace}
Let $\mu$ denote an invariant measure on $E$.  Let $\tilde{\mu}$ denote the probability measure on $\partial E$ given by Proposition \ref{prop:BoundaryMeasure}.  There is a representation $C_0(\partial E) \rightarrow L^2(\partial E, \tilde{\mu})$ given by multiplication.  Now define
\[ L^2(E, \mu) := L^2(E) \otimes_{C_0(\partial E)} L^2(\partial E, \tilde{\mu}). \]
Then $\pi$ induces a representation $\pi_\mu : \cstar(E) \rightarrow \mathbb{B}(L^2(E))$ given by $\pi_\mu(x) = \pi(x) \otimes \operatorname{id}$ for each $x \in \cstar(E)$.
\end{definition}

Note that the gauge action $\gamma'$ of $\mathbb{T}$ on $\mathbb{B}(L^2(E))$ induces an action, also denoted $\gamma'$ of $\mathbb{T}$ on $\mathbb{B}(L^2(E, \mu))$.  Also, the representation $\pi_\mu$ is gauge invariant.

\begin{remark}
As mentioned in the introduction, the spaces $L^2(E)$ and $L^2(E, \mu)$ can be viewed as fibred objects over $\partial E$.  This will not be needed in the paper, but we give an outline of how this can be done.

For each $\beta \in \partial E$, define $H(\beta) := \ell^2(d^{-1}(\beta))$ and let $\prod_{\beta \in \partial E} H(\beta)$ denote the $\ell^\infty$-product of the spaces $H(\beta)$.  The map
\[ C_c(\partial E) \rightarrow \prod_{\beta \in \partial E} H(\beta) \qquad f \mapsto (f|d^{-1}(\beta))_{\beta \in \partial E} \]
is linear and isometric.  Hence it extends to a linear isometry on $L^2(E)$.  There is a unique continuous field of Hilbert spaces $\mathcal{H}$ over $\partial E$ in the sense of Chapter 10 of \cite{Dixmier} such that the fibre over $\beta$ is $H(\beta)$ for every $\beta \in \partial E$.  Moreover, given a vector field $\xi \in \mathcal{H}$, $\xi \in L^2(E)$ if and only if the function $\beta \mapsto \|\xi(\beta)\|$ vanishes at infinity.

If $\mu$ is an invariant measure on $E$ and $\tilde{\mu}$ is the induced measure on $\partial E$, there is a unique measurable field of Hilbert spaces $\beta \mapsto H(\beta)$ containing $\mathcal{H}$.  Moreover, there is a unitary
\[ L^2(E, \mu) \rightarrow \int^\oplus_{\partial E} H(\beta) \, d\tilde{\mu}(\beta) \]
given by restricting $f \in L^2(E, \mu)$ to the fibres of $d : X \rightarrow \partial E$ as above.  Conjugating $\pi_\mu$ by the unitary also yields a direct integral decomposition for the representation $\pi_\mu$.
\end{remark}

\section{Traces on Topological Graph Algebras}

We now show every invariant measure $\mu$ on a topological graph $E$ induces a gauge invariant tracial state on $\cstar(E)$.  The strategy is to construct of vector $\hat{u} \in L^2(E, \mu)$ such that $\langle \hat{u}, \pi_\mu(\cdot) \hat{u} \rangle$ is a tracial state on $\cstar(E)$.

Define a function $\iota : \partial E \rightarrow X$ by $\iota(\alpha) = (\alpha, 0, \alpha)$ and note that $\iota$ is a homeomorphism onto a closed and open subset of $X$, as follows easily from the characterization of convergent nets in $\partial E$ given in Lemma \ref{lem:ConvergentNetsInX}.  We will identify $\partial E$ as a subset of $X$ through the map $\iota$.

Fix an increasing net of compact sets $K_i \subseteq \partial E$ with union $\partial E$.  Let $f_i \in C_c(\partial E) \subseteq L^2(E)$ be such that $0 \leq f_i \leq 1$ and $f_i|K_i = 1$.  Define
\[ \hat{u}_i = f_i \otimes 1 \in L^2(E) \otimes_{C_0(\partial E)} L^2(\partial E, \tilde{\mu}) = L^2(E, \mu). \]
Fix $\varepsilon > 0$.  Since $\tilde{\mu}$ is a Radon measure on $\partial E$, there is $i_0$ such that $\tilde{\mu}(\partial E \setminus K_{i_0}) < \varepsilon$.  Now, for $i, j \geq i_0$, we have
\[ \| \hat{u}_i - \hat{u}_j \|^2 = \int_{\partial E} |f_i - f_j|^2 \, d\tilde{\mu} \leq \tilde{\mu}(\partial E \setminus K_{i_0}) < \varepsilon. \]
So $\hat{u}_i$ is a Cauchy net in $L^2(E, \mu)$ and hence converges to $\hat{u} \in L^2(E, \mu)$.

Note that $\hat{u}$ is a unit vector.  Hence composing the vector state defined by $\hat{u}$ with the representation $\pi_\mu : \cstar(E) \rightarrow \mathbb{B}(L^2(E, \mu))$ yields a state on $\cstar(E)$ given by
\begin{equation}\label{eqn:TraceFormula}
  \langle \hat{u}, \pi_\mu(x) \hat{u} \rangle = \lim_{i} \int_{\partial E} \langle f_i, \pi(x) f_i \rangle(\alpha) \, d\tilde{\mu}(\alpha)
\end{equation}
for any $x \in \cstar(E)$, for any net $(f_i) \subseteq C_c(\partial E)$ increasing pointwise to 1.

We will show that this state is tracial in Theorem \ref{thm:InvariantTrace}.  First we need a couple lemmas.

\begin{lemma}\label{lem:TracesOnCuntzPimsnerAlgebras}
Let $H$ be a C*-correspondence over a C*-algebra $A$.  Let $(\pi^0, \pi^1)$ denote the universal covariant representation on the Cuntz-Pimsner algebra $\mathcal{O}_A(H)$, and let $\pi^k$ denote the induced maps $H^{\otimes k} \rightarrow \mathcal{O}_A(H)$.  If $\tau$ is a gauge invariant state on $\mathcal{O}_A(H)$ such that $\tau(\pi^k(\xi)^* \pi^k(\eta)) = \tau(\pi^k(\eta) \pi^k(\xi)^*)$ for all $\xi, \eta \in H^{\otimes k}$ and $k \geq 0$, then $\tau$ is a tracial state on $\mathcal{O}_A(H)$.
\end{lemma}

\begin{proof}
The elements of the form $\pi^k(\xi) \pi^\ell(\eta)^*$ for $k, \ell \in \mathbb{N}$, $\xi \in H^{\otimes k}$ and $\eta \in H^{\otimes \ell}$ span a dense subset of $\mathcal{O}_A(H)$.  Hence to show $\tau$ is a tracial state on $\mathcal{O}_A(H)$, it's enough to show
\[ \tau(x \pi^1(\zeta)) = \tau(\pi^1(\zeta) x) \qquad \text{and} \qquad \tau(x \pi^1(\zeta)^*) = \tau(\pi^1(\zeta)^* x) \]
for all $x \in \mathcal{O}_A(H)$ and $\zeta \in H$.  In fact, the second equation follows from the first by taking adjoints.  So it's enough to verify the first equation.

Fix $\zeta \in H$.  Suppose $k, \ell \in \mathbb{N}$, $\xi \in H^{\otimes k}$, and $\eta \in H^{\otimes \ell}$.  If $\ell \neq k + 1$, then
\[ \tau(\pi^k(\xi)\pi^\ell(\eta)^*\pi^1(\zeta)) = \tau(\pi^1(\zeta)\pi^k(\xi)\pi^\ell(\eta)^*) = 0 \]
since $\tau$ is gauge invariant.  Hence to verify $\tau$ is a tracial state, it's enough to show
\[ \tau(\pi^k(\xi)\pi^{k+1}(\eta \otimes \eta')^* \pi^1(\zeta)) = \tau(\pi^1(\zeta)\pi^k(\xi)\pi^{k+1}(\eta \otimes \eta')^*) \]
for all $k \in \mathbb{N}$, $\xi, \eta' \in H^{\otimes k}$, and $\eta \in H$.  We now verify this equality.
\begin{align*}
     \tau(\pi^k(\xi)\pi^{k+1}&(\eta \otimes \eta')^* \pi^1(\zeta)))
   = \tau(\pi^k(\xi)\pi^k(\eta')^*\pi^1(\eta)^*\pi^1(\zeta)) \\
  &= \tau(\pi^k(\xi)\pi^k(\eta')^*\pi^0(\langle \eta, \zeta \rangle))
   = \tau(\pi^k(\xi)\pi^k(\langle \zeta, \eta \rangle \eta')^*) \\
  &= \tau(\pi^k(\langle \zeta, \eta \rangle \eta')^*\pi^k(\xi))
   = \tau(\pi^k(\eta')^*\pi^0(\langle \eta, \zeta \rangle)\pi^k(\xi)) \\
  &= \tau(\pi^k(\eta')^*\pi^1(\eta)^*\pi^1(\zeta)\pi^k(\xi))
   = \tau(\pi^{k+1}(\eta \otimes \eta')^*\pi^{k+1}(\zeta \otimes \xi)) \\
  &= \tau(\pi^{k+1}(\zeta \otimes \xi)\pi^{k+1}(\eta \otimes \eta')^*)
   = \tau(\pi^1(\zeta)\pi^k(\xi)\pi^{k+1}(\eta \otimes \eta')^*).
\end{align*}
Hence $\tau$ is a tracial state.
\end{proof}

\begin{lemma}\label{lem:ConvergesOfTracesOnCuntzPimsnserAlgebras}
Let $H$ be a C*-correspondence over a C*-algebra $A$.  Let $(\pi^0, \pi^1)$ denote the universal covariant representation on the Cuntz-Pimsner algebra $\mathcal{O}_A(H)$.  If $(\tau_i)$ is a net of gauge invariant tracial states on $\mathcal{O}_A(H)$ and $\tau$ is a gauge invariant tracial on $\mathcal{O}_A(H)$, then $\tau_i \rightarrow \tau$ weak* if and only if $\tau_i(\pi^0(a)) \rightarrow \tau(\pi^0(a))$ for every $a \in A$.
\end{lemma}

\begin{proof}
The forward implication is trivial.  Assume $\tau_i \circ \pi^0 \rightarrow \tau \circ \pi^0$ weak*.  If $k, \ell \in \mathbb{N}$ with $k \neq \ell$, $\xi \in H^{\otimes k}$ and $\ell \in H^{\otimes \ell}$, then
\[ \tau_i(\pi^k(\xi) \pi^\ell(\eta)^*) = \tau(\pi^k(\xi) \pi^\ell(\eta)^*) = 0. \]
If $k \in \mathbb{N}$ and $\xi, \eta \in H^{\otimes k}$, then
\[ \tau_i(\pi^k(\xi) \pi^k(\eta)^*) = \tau_i(\pi^0(\langle \eta, \xi \rangle)) \rightarrow \tau(\pi^0(\langle \eta, \xi \rangle)) = \tau(\pi^k(\xi) \pi^k(\eta)^*). \]
Since the elements of the form $\pi^k(\xi) \pi^\ell(\eta)^*$ for $k, \ell \in \mathbb{N}$, $\xi \in H^{\otimes k}$, and $\eta \in H^{\otimes \ell}$, span a dense subset of $\mathcal{O}_A(H)$, it follows that $\tau_i \rightarrow \tau$ weak*.
\end{proof}

Let $T(\cstar(E))$ denote the space of tracial states on $\cstar(E)$ equipped with the weak* topology and let $T(\cstar(E)^\gamma$ denote the subspace consisting of tracial states invariant under the gauge action $\gamma : \mathbb{T} \curvearrowright \cstar(E)$.  Also, let $T(E)$ denote the space of invariant measures on $E$ equipped with the weak* topology.

\begin{theorem}\label{thm:InvariantTrace}
Given $\tau \in T(\cstar(E))$, there is a unique $\mu \in T(E)$ such that
\begin{equation}\label{eqn:TraceToMeasure}
  \tau(a) = \int_{E^{^0}} a(v) \, d\mu(v)
\end{equation}
for all $a \in C_0(E^0) \subseteq \cstar(E)$.  Conversely, given $\mu \in T(E)$, there is a unique $\tau \in T(\cstar(E))^\gamma$ such that
\begin{equation}\label{eqn:MeasureToTrace}
  \tau(x) = \langle \hat{u}, \pi_\mu(x) \hat{u} \rangle_{L^2(E, \mu)}
\end{equation}
for each $x \in \cstar(E)$.  These maps define inverse affine homeomorphisms between $T(E)$ and $T(\cstar(E))^\gamma$.
\end{theorem}

\begin{proof}
Let $(\pi_u^0, \pi_u^1)$ denote the universal covariant representation of $E$ on $\cstar(E)$.

Given $\tau \in T(\cstar(E))$, the composition $C_0(E^0) \overset{\pi^0_u}{\longrightarrow} \cstar(E) \overset{\tau}{\longrightarrow} \mathbb{C}$ is a state on $C_0(E^0)$ and hence there is a unique Radon probability measure $\mu$ on $E^0$ satisfying \eqref{eqn:TraceToMeasure}.  We need to show the invariance of $\mu$.

Let $\varphi : \mathbb{K}(H(E)) \rightarrow \cstar(E)$ denote the $\ast$-homomorphism induced by $(\pi_u^0, \pi_u^1)$; that is, $\varphi(\xi \otimes \eta^*) = \pi^1_u(\xi) \pi^1_u(\eta)^*$ for $\xi, \eta \in H(E)$.  Note that since $\tau$ is a tracial state, for all $\xi, \eta \in H(E)$ we have
\begin{equation}\label{eqn:InnerProductVsRankOne}
\tau(\pi^0_u(\langle \eta, \xi \rangle)) = \tau(\pi^1_u(\eta)^*\pi^1(\xi)) = \tau(\pi^1_u(\xi)\pi^1(\eta^*)) = \tau(\varphi(\xi \otimes \eta^*)).
\end{equation}

Suppose $a \in C_c(E^0)$ and fix $\zeta \in C_c(E^1)$ with $0 \leq \zeta \leq a \circ r$ and let $\lambda : C_b(E^1) \rightarrow \mathbb{B}(H(E))$ be the multiplication map as in Proposition \ref{prop:LeftMultiplicationMap}.  By Lemma \ref{lem:PartitionOfUnity}, there are $\xi_i, \eta_i \in C_c(E^1)$ for $i = 1, \ldots, n$ with $\sum_i \xi_i \otimes \eta_i^* = \lambda(\zeta)$ and $\sum_i \xi_i \overline{\eta_i} = \zeta$.  By Lemma \ref{lem:InequalityForFunctionsOnEdges}, $\varphi(\lambda(\zeta)) \leq \pi^0_u(a)$ and now we compute
\begin{align*}
 \int_{E^{^1}} \zeta \, ds^*\mu &= \int_{E^{^0}} \sum_{s(e) = v} \zeta(e) \, d\mu(v) = \sum_{i=1}^n \int_{E^{^0}} \sum_{s(e) = v} \xi(e) \overline{\eta(e)} \, d\mu(v) \\
 &= \sum_{i=1}^n \int_{E^{^0}} \langle \eta, \xi \rangle \, d\mu = \sum_{i=1}^n \tau(\pi^0_u(\langle \eta, \xi \rangle)) \overset{\eqref{eqn:InnerProductVsRankOne}}{=} \sum_{i=1}^n \tau(\varphi(\xi_i \otimes \eta_i^*)) \\ &= \tau(\varphi(\lambda(\zeta))) \leq \tau(\pi^0_u(a)) = \int_{E^{^0}} a \, d\mu
\end{align*}
Now the Monotone Convergence Theorem implies
\[ \int_{E^{^1}} a \circ r \, ds^*\mu \leq \int_{E^{^0}} a \, d\mu. \]
Moreover, when $a \in C_c(E^0_\reg)$, we may take $\zeta = a \circ r$ in the calculation above.  Then $\varphi(\lambda(\zeta)) = \pi^0_u(a)$ by the covariance condition in the definition of $\cstar(E)$.  Repeating the calculation above yields
\[ \int_{E^{^1}} a \circ r \, ds^*\mu = \int_{E^{^0}} a \, d\mu \]
which proves $\mu$ is an invariant measure on $E$.

Conversely, suppose $\mu \in T(E)$ is an invariant measure.  Then $\mu$ defines a state $\tau$ on $\cstar(E)$ via the formula given in \eqref{eqn:MeasureToTrace}.  The representation $\pi : \cstar(E) \rightarrow \mathbb{B}(L^2(E))$ is gauge invariant by Theorem \ref{thm:GaugeActionOnBL2}.  Moreover, with the unitaries $U_z$ on $L^2(E)$ as in Definition \ref{defn:GaugeActionOnBL2}, for every $f \in C_c(\partial E)$, $U_zf = f$.  Hence for $x \in \cstar(E)$, $z \in \mathbb{T}$, and $f \in C_c(\partial E)$,
\[ \langle f, \pi(\gamma_z(x))f \rangle = \langle f, U_z \pi(x) U_z^* f \rangle = \langle f, \pi(x)f \rangle. \]
It follows from \eqref{eqn:TraceFormula} that $\tau(\gamma_z(x)) = \tau(x)$ for every $x \in \cstar(E)$ and $z \in \mathbb{T}$; that is, $\tau$ is gauge invariant.

We prove $\tau$ is tracial.  Fix an integer $k \geq 1$ and suppose $\xi, \eta \in C_c(E^k)$.  Define
\[ K := \{ \alpha \in \partial E : r(\alpha) \in s(\operatorname{supp}(\xi)) \} \cup \{\alpha \in \partial E \setminus \partial E_{k-1} : \alpha_1 \ldots \alpha_k \in \operatorname{supp}(\xi) \} \]
and note that $K$ is compact by Proposition \ref{prop:BoundaryProjectiveLimit}.  Let $f_0 \in C_c(\partial E)$ be such that $0 \leq f_0 \leq 1$ and $f_0|K = 1$ and fix $f \in C_c(\partial E)$ with $f_0 \leq f \leq 1$.

Given a vertex $v \in E^0$, write $E^k v$ for all paths $\alpha \in E^k$ with $s(\alpha) = v$.  Fix $\beta \in \partial E$.  Since $f$ is supported on $\partial E \subseteq X$, we have
\begin{align*}
     \langle f, \pi^k(\xi)^* \pi^k(\eta)(f) \rangle(\beta)
  &= \sum_{d(x) = \beta} \overline{f(x)} \pi^k(\xi)^*\pi^k(\eta)(f)(x) \\
  &= f(\beta, 0, \beta) \pi^k(\xi)^* \pi^k(\eta)(f)(\beta, 0, \beta) \\
  &= \sum_{\alpha \in E^k r(\beta)} \overline{\xi(\alpha)} \eta(\alpha) f(\beta, 0, \beta)^2 \\
  &= \sum_{\alpha \in E^k r(\beta)} \overline{\xi(\alpha)} \eta(\alpha).
\end{align*}
Similarly, if $|\beta| \geq k$,
\begin{align*}
\langle f, \pi^k(\eta) \pi^k(\xi)^*(f) \rangle(\beta) &= f(\beta, 0, \beta) \eta(\beta_1 \ldots \beta_k)\sum_{\alpha \in E^k s(\beta_k)} \overline{\xi(\alpha)} f(\alpha\sigma(\beta), 0, \beta) \\ &= \eta(\beta_1 \cdots \beta_k) \overline{\xi(\beta_1 \cdots \beta_k)} f(\beta, 0, \beta)^2 = \overline{\xi(\beta_1 \cdots \beta_k)} \eta(\beta_1 \cdots \beta_k),
\end{align*}
and if $|\beta| < k$, then $\langle f, \pi^k(\eta) \pi^k(\xi)^*(f) \rangle(\beta) = 0$.

Now, by the invariance of the measure $\tilde{\mu}$ given in Corollary \ref{cor:BoundaryMeasureInvariant}, we have
\begin{align*}
 \int_{\partial E} \langle f, &\pi^k(\eta) \pi^k(\xi)^*(f) \rangle \, d\tilde{\mu} = \int_{\partial E \setminus \partial E_{k-1}} \overline{\xi(\beta_1\cdots\beta_k)} \eta(\beta_1\cdots\beta_k) \, d\tilde{\mu} \\ &\overset{\eqref{eqn:BoundaryMeasureInvariant'}}{=} \int_{\partial E} \sum_{\alpha \in E^k r(\beta)} \overline{\xi(\alpha)} \eta(\alpha) \, d\tilde{\mu}(\beta) = \int_{\partial E} \langle f, \pi^k(\xi)^* \pi^k(\eta)(f) \rangle(\beta).
\end{align*}
Now by \eqref{eqn:TraceFormula}, $\tau(\pi^k_u(\eta) \pi^k_u(\xi)^*) = \tau(\pi^k_u(\xi)^* \pi^k_u(\eta))$ and Lemma \ref{lem:TracesOnCuntzPimsnerAlgebras} implies $\tau$ is a tracial state.  Hence $\tau \in T(\cstar(E))^\gamma$.

We now show the maps constructed above are inverses.  Starting with an invariant measure $\mu$ on $E$, let $\tau$ denote the tracial state on $\cstar(E)$ defined by \eqref{eqn:MeasureToTrace} and let $\mu'$ denote the invariant measure induced by $\tau$ as in \eqref{eqn:TraceToMeasure}.  Fix $a \in C_c(E^0)$ and let $f \in C_c(\partial E)$ be such that $0 \leq f \leq 1$ and $f(\beta) = 1$ for all $\beta \in \partial E$ with $r(\beta) \in \operatorname{supp}(a)$.  Then
\begin{align*}
   \int_{E^{^0}} a \, d\mu' &\overset{\eqref{eqn:TraceToMeasure}}{=} \tau(\pi^0_u(a)) \overset{\eqref{eqn:MeasureToTrace}}{=} \int_{\partial E} \langle \hat{u}, \pi^0(a) \hat{u} \rangle \, d\tilde{\mu} \overset{\eqref{eqn:TraceFormula}}{=} \int_{\partial E} \langle f, \pi^0(a) f \rangle \, d\tilde{\mu} \\ &= \int_{\partial E} f(\alpha) a(r(\alpha)) f(\alpha) \, d\tilde{\mu}(\alpha) = \int_{\partial E} a \circ r \, d\tilde{\mu} \overset{\eqref{eqn:BoundaryMeasureRange}}{=} \int_{E^{^0}} a \, d \mu.
\end{align*}
Hence $\mu = \mu'$.

Now suppose $\tau$ is a gauge invariant tracial state on $\cstar(E)$.  Let $\mu$ be the invariant measure on $E$ defined by \eqref{eqn:TraceToMeasure} and let $\tau'$ be the tracial state on $\cstar(E)$ defined by \eqref{eqn:MeasureToTrace}.  Arguing as in the previous paragraph, we have $\tau(\pi^0_u(a)) = \int_{E^{^0}} a \, d\mu = \tau'(\pi^0_u(a))$ for each $a \in C_0(E^0)$.  Now, for $k \in \mathbb{N}$ and $\xi, \eta \in C_c(E^k) \subseteq H(E)^{\otimes k}$, we have
\[ \tau(\pi^k_u(\xi) \pi^k_u(\eta)^*) = \tau(\pi^k_u(\eta)^* \pi^k_u(\xi)) = \tau(\pi^0_u(\langle \eta, \xi \rangle)) = \tau'(\pi^0_u(\langle \xi, \eta \rangle)) = \tau'(\pi^k_u(\xi) \pi^k_u(\eta)^*), \]
since $\tau$ and $\tau'$ are both tracial states.  Moreover, since $\tau$ and $\tau'$ are both gauge invariant, if $k, \ell \in \mathbb{N}$ are distinct integers, $\xi \in C_c(E^k)$ and $\eta \in C_c(E^\ell)$,
\[ \tau(\pi^k_u(\xi) \pi^\ell_u(\eta)^*) = 0 = \tau(\pi^k_u(\xi) \pi^\ell_u(\eta)^*). \]
Hence $\tau = \tau'$.

It is clear from the construction that the map $T(\cstar(E))^\gamma \rightarrow T(E)$ is affine and hence the inverse map is affine.  Suppose $\tau_i, \tau$ are gauge invariant tracial states on $\cstar(E)$ and $\mu_i, \mu$ are the corresponding invariant measure on $\cstar(E)$.  Then $\mu_i \rightarrow \mu$ if and only if $\tau_i \circ \pi^0 \rightarrow \tau \circ \pi^0$.  By Lemma \ref{lem:ConvergesOfTracesOnCuntzPimsnserAlgebras}, $\mu_i \rightarrow \mu$ if and only if $\tau_i \rightarrow \tau$.  Hence the map $T(\cstar(E))^\gamma \rightarrow T(E)$ is a homeomorphism.
\end{proof}

The following result shows the representation associated to an invariant measure $\mu$ on $E$ constructed in \ref{defn:GNSSpace} is precisely the GNS representation of the gauge invariant trace $\tau$ on $\cstar(E)$ corresponding to $\mu$.

\begin{theorem}
If $\tau$ is a gauge invariant trace on $\cstar(E)$ and $\mu$ is the corresponding invariant measure on $E$ given by \eqref{eqn:TraceToMeasure}, the then representations
\[ \pi_\tau : \cstar(E) \rightarrow \mathbb{B}(L^2(\cstar(E), \tau)) \qquad \text{and} \qquad \pi_\mu : \cstar(E) \rightarrow \mathbb{B}(L^2(E, \mu)) \]
are unitarily equivalent.
\end{theorem}

\begin{proof}
Fix an increasing net $K_i$ of compact subsets of $\partial E \subseteq X$ whose interiors cover $\partial E$.  Let $f_i \in C_c(\partial E)$ be such that $0 \leq f_i \leq 1$ and $f|K_i = 1$.  Given $k, \ell \in \mathbb{N}$, $\xi \in C_c(E^k)$, and $\eta \in C_c(E^\ell)$, the net $\pi^k(\xi) \pi^\ell(\eta)^*(f_i)$ in $C_c(X)$ is eventually constant.  Abusing notation slightly, we let $\pi^k(\xi) \pi^\ell(\eta)^* (\chi_{\partial E})$ denote the limit in $C_c(X)$.

Consider the set
\[ A = \operatorname{span}\, \{ \pi^k(\xi) \pi^\ell(\eta)^* (\chi_{\partial E}) : k, \ell \in \mathbb{N}, \xi \in C_c(E^k), \eta \in C_c(E^\ell) \} \subseteq C_c(X). \]
We will show $A$ is uniformly dense in $C_c(X)$.  Fix $k, \ell \in \mathbb{N}$, $\xi \in C_c(E^k)$, $\eta \in C_c(E^\ell)$, and $(\alpha, n, \beta) \in X$. If $k - \ell = n$ and there are $\alpha' \in E^k$, $\beta' \in E^\ell$, and $\omega \in \partial E$ with $\alpha' \omega = \alpha$ and $\beta' \omega = \beta$, then
\[ \pi^k(\xi) \pi^\ell(\eta)^*(\chi_{\partial E}) = \xi(\alpha') \overline{\eta(\beta')}; \]
otherwise, $\pi^k(\xi) \pi^\ell(\eta)^*(\chi_{\partial E}) = 0$.  It follows that for $k, k', \ell, \ell' \in \mathbb{N}$, $\xi \in C_c(E^k), \xi' \in C_c(E^{k'})$, $\eta \in C_c(E^\ell)$, and $\eta' \in C_c(E^{\ell'})$,
\[ \pi^k(\xi) \pi^\ell(\eta)^*(\chi_{\partial E}) \pi^{k'}(\xi') \pi^{\ell'}(\eta')^*(\chi_{\partial E}) = \begin{cases} \pi^k(\xi \xi') \pi^\ell(\eta \eta')^* (\chi_{\partial E}) & k = k', \ell = \ell' \\ 0 & \text{else} \end{cases} \]
and
\[ \pi^k(\xi) \pi^\ell(\eta)^*(\chi_{\partial E}) = \pi^k(\overline{\xi}) \pi^\ell(\overline{\eta})^*(\chi_{\partial E}). \]
Hence $A$ is a $\ast$-subalgebra of $C_c(X)$.

Given $(\alpha, n, \beta) \in X$, there are $k, \ell \in \mathbb{N}$, $\alpha' \in E^k$, $\beta' \in E^\ell$, and $\omega \in \partial E$ such that $k - \ell = n$, $\alpha = \alpha' \omega$, and $\beta = \beta' \omega$.  Fix $\xi \in C_c(E^k)$ and $\eta \in C_c(E^\ell)$ such that $\xi(\alpha') = 1$ and $\eta(\beta') = 1$.  Then
\[ \pi^k(\xi) \pi^\ell(\eta)^*(\chi_{\partial E})(\alpha, n, \beta) = 1. \]
So the algebra $A$ does not vanish at any point in $X$.  Similar considerations show $A$ separates points in $X$.  Hence the Stone-Weierstrass Theorem implies $A$ is uniformly dense in $C_c(X)$.

Since $C_c(X)$ is dense in $L^2(E)$, we have $A$ is dense in $L^2(E)$ and in particular, the vectors $\pi_\mu^k(\xi) \pi_\mu^\ell(\eta)^* \hat{u}$ span a dense set of subspace of $L^2(E, \mu)$.  This shows $\hat{u}$ is a cyclic vector for the representation $\pi_\mu$.  The result follows from \eqref{eqn:MeasureToTrace} and the uniqueness of the GNS representation.
\end{proof}

We would like to mention two special cases of Theorem \ref{thm:InvariantTrace}.  If $X$ is a locally compact Hausdorff space and $\sigma$ is a homeomorphisms of $X$, Theorem \ref{thm:InvariantTrace} reduces to the classical result that every invariant probability measure on $(X, \sigma)$ induces a tracial state on the crossed product $C_0(X) \rtimes_\sigma \mathbb{Z}$.  When $E$ is a discrete graph, our notion of invariant measures on $E$ agrees with Tomforde's notion of graph traces on $E$ introduced in \cite{Tomforde:OrderedKTheory} and we recover the results of Section 3.3 in \cite{Tomforde:OrderedKTheory}.

For crossed products, if $\sigma$ is a free action of $\mathbb{Z}$ on a locally compact Hausdorff space $X$, then every tracial state on $C_0(X) \rtimes \mathbb{Z}$ is gauge invariant.  If $E$ is a discrete graph satisfying condition (K), then every tracial state on $\cstar(E)$ is gauge invariant as noted in Section 3.3 of \cite{Tomforde:OrderedKTheory}.  It is very likely that these results have have an analogue for topological graphs, but we were unable to prove this.  Freeness of topological graphs was defined by Katsura in \cite[Definition 7.2]{Katsura:TGA3} and gives a simultaneous generalization of the notions of freeness for homeomorphisms and condition (K) for discrete graphs.

\begin{conjecture}\label{conj:Freeness}
If $E$ is a free topological graph, then every tracial on $\cstar(E)$ is gauge invariant.
\end{conjecture}

In particular, if $E$ is a topological graph such that $\cstar(E)$ is simple, then $E$ is free by \cite[Theorem 8.12]{Katsura:TGA3}.  Hence if the conjecture above holds, Theorem \ref{thm:InvariantTrace} would yield a complete description of the tracial state simplex for simple topological graph \cstar-algebras.  We prove a special case of Conjecture \ref{conj:Freeness} in Corollary \ref{cor:NoCycles} below.

\begin{lemma}\label{lem:TracesOnCuntzPimsnerAlgebras2}
Let $H$ be a C*-correspondence over a C*-algebra $A$, and let $\pi^k : H^{\otimes k} \rightarrow \cstar(E)$ denote the canonical map for $k \geq 0$.  If $\tau$ is a tracial state on $\mathcal{O}_A(H)$ and $\tau(\pi^k(\xi)) = 0$ for all $\xi \in H^{\otimes k}$ and $k \geq 1$, then $\tau$ is gauge invariant.
\end{lemma}

\begin{proof}
Suppose $k, \ell \in \mathbb{N}$ with $k > \ell$.  Let $\xi, \eta \in H^{\otimes k}$ and $\xi' \in H^{\otimes(k - \ell)}$ be given.  Then
\[ \tau(\pi^k(\xi \otimes \xi') \pi^\ell(\eta)^*) = \tau(\pi^k(\eta)^* \pi^k(\xi) \pi^{k - \ell}(\xi')) = \tau(\pi^{k - \ell}(\langle \eta, \xi \rangle \xi')) = 0. \]
It follows that if $k, \ell \in \mathbb{N}$ are distinct, $\xi \in H^{\otimes k}$ and $\eta \in H^{\otimes \ell}$, then $\tau(\pi^k(\xi) \pi^{\ell}(\eta)^*) = 0$.  Indeed, when $k > \ell$, this is an immediate consequence of the calculation above, and when $\ell > k$, the result follows since by taking adjoints.

Now let $\gamma : \mathbb{T} \curvearrowright \mathcal{O}_A(H)$ denote the gauge action.  For $k, \ell \in \mathbb{N}$, $\xi \in H^{\otimes k}$, and $\eta \in H^{\otimes \ell}$, and $z \in \mathbb{T}$,
\[ \tau(\gamma_z(\pi^k(\xi) \pi^{\ell}(\eta)^*)) = z^{k - \ell} \tau(\pi^k(\xi) \pi^\ell(\eta)^*) = \tau(\pi^k(\xi) \pi^\ell(\eta)^*). \]
Since the elements of the form $\pi^k(\xi) \pi^\ell(\eta)^*$ span a dense subspace of $\mathcal{O}_A(H)$, $\tau$ is gauge invariant.
\end{proof}

A path $\alpha \in E^*$ is called a \emph{cycle} if $|\alpha| \geq 1$ and $s(\alpha) = r(\alpha)$.

\begin{proposition}\label{prop:NoCycles}
Suppose $\tau$ is a tracial state on a topological graph algebra $\cstar(E)$ and $\mu$ is the invariant measure on $E$ induced by $\tau$ as in \eqref{eqn:TraceToMeasure}.  If for every cycle $\alpha \in E^*$, $s(\alpha) \notin \operatorname{supp}(\mu)$, then $\tau$ is gauge invariant.
\end{proposition}

\begin{proof}
Let $\tau$ be a tracial state on $\cstar(E)$.  Fix $k \geq 1$ and let $\xi \in C_c(E^k)$ be given.  By Lemma \ref{lem:TracesOnCuntzPimsnerAlgebras2}, it's enough to show $\tau(\pi^k(\xi)) = 0$.

If $\alpha \in E^k$ with $s(\alpha) \neq r(\alpha)$, then there are disjoint open neighborhoods $U_\alpha, V_\alpha \subseteq E^0$ of $s(\alpha)$ and $r(\alpha)$, respectively.  Now, $s^{-1}(U_\alpha) \cap r^{-1}(V_\alpha) \subseteq E^k$ is an open neighborhood of $\alpha$ and hence contains a compact neighborhood $K_\alpha$ of $\alpha$.  If $\alpha \in E^k$ and $s(\alpha) = r(\alpha)$, then $s(\alpha) \notin \operatorname{supp}(\mu)$.  As $\operatorname{supp}(\mu)$ is closed, there is a compact neighborhood $K_\alpha \subseteq E^k$ of $\alpha$ with $s(K_\alpha) \cap \operatorname{supp}(\mu) = \emptyset$.

The interiors of the sets $K_\alpha$, $\alpha \in E^k$, form an open cover of $E^k$.  Fix $\alpha(1), \ldots, \alpha(n) \in E^k$ such that the interiors of the $K_{\alpha(i)}$ form an open cover of the compact set $\operatorname{supp}(\xi) \subseteq E^k$.  There are functions $\xi_i \in C_c(E^k)$ such that $\sum_i \xi_i = \xi$ and each $\xi_i$ is supported in the interior on $K_{\alpha(i)}$.  To show $\tau(\pi^k(\xi)) = 0$, it's enough to show $\tau(\pi^k(\xi_i)) = 0$ for all $i$.

If $s(\alpha(i)) \neq r(\alpha(i))$, the sets $s(K_{\alpha(i)})$ and $r(K_{\beta(i)})$ are disjoint compact subsets of $E^0$.  Hence there is an $a \in C_c(E^0)$ with $a_i|s(K_{\alpha(i)}) = 1$ and $a|r(K_{\alpha(i)}) = 0$.  Now, $\xi_i a = \xi_i$ and $a \xi_i = 0$.  Thus
\[ \tau(\pi^k(\xi_i)) = \tau(\pi^k(\xi_i)\pi^0(a_i)) = \tau(\pi^0(a_i) \pi^k(\xi_i)) = 0. \]

If $s(\alpha(i)) = r(\alpha(i))$, then $s(K_{\alpha(i)})$ is a compact subset of $E^0$ disjoint from $\operatorname{supp}(\mu)$.  There is a positive function $a \in C_c(E^0)$ with $a_i|s(K_{\alpha(i)}) = 1$ and $a|\operatorname{supp}(\mu) = 0$.  Then $\xi_i a^{1/2} = \xi_i$ and $\tau(\pi^0(a)) = \int a \, d\mu = 0$.  Now by the Cauchy-Schwartz inequality,
\[ |\tau(\pi^k(\xi_i))| = |\tau(\pi^k(\xi_i)\pi^0(a)^{1/2})| = \tau(\pi^k(\xi_i)^*\pi^k(\xi_i))^{1/2} \tau(\pi^0(a))^{1/2} = 0. \]
Hence $\tau(\pi^k(\xi_i)) = 0$.

We have shown $\tau(\pi^k(\xi_i)) = 0$ for all $i = 1, \ldots, n$.  Hence $\tau(\pi^k(\xi)) = 0$.  By Lemma \ref{lem:TracesOnCuntzPimsnerAlgebras2}, $\tau$ is gauge invariant.
\end{proof}

\begin{corollary}\label{cor:NoCycles}
If $E$ has no cycles, every tracial state on $\cstar(E)$ is gauge invariant.
\end{corollary}

To prove Conjecture \ref{conj:Freeness}, it is enough to show if $\mu$ is an invariant measure on $E$, $\alpha$ is a cycle in $E$, and $s(\alpha) \in \operatorname{supp}(\mu)$, then $s(\alpha)$ is periodic in the sense of \cite[Definition 7.1]{Katsura:TGA3}.  Indeed, by definition, a free topological graph has no periodic vertices and hence there would be no cycles with source in $\operatorname{supp}(\mu)$.  Then Proposition \ref{prop:NoCycles} would prove Conjecture \ref{conj:Freeness}.  For a discrete graphs, there is an easy direct proof of this is given below in Corollary \ref{cor:ConditionK}.  For topological graphs, a similar approach should work, but we could not overcome the topological technicalities.

Using our result Proposition \ref{prop:NoCycles}, we can recover the two special cases of Conjecture \ref{conj:Freeness} mentioned above.  Recall a cycle $\alpha = \alpha_1 \ldots \alpha_n \in E^n$ is called \emph{simple} if $\alpha_i \neq \alpha_n$ for each $1 \leq i < n$.  A discrete graph $E$ satisfies condition (K), if there is no vertex $v \in E^0$ which is the source of a unique simple cycle.

\begin{corollary}
If $\sigma : \mathbb{Z} \curvearrowright X$ is a free action on a locally compact Hausdorff space $X$, then every tracial state on $C_0(X) \rtimes_\sigma \mathbb{Z}$ is gauge invariant.
\end{corollary}

\begin{proof}
Viewing $(X, \sigma)$ as a topological graph $E$ with $E^0 = E^1 = X$, $s = \sigma$, and $r = \operatorname{id}$, a vertex $v \in E^0$ is the source of a cycle if and only if $v$ is a periodic point for the action $\sigma$.  Since the action is free, $E$ has no cycles.
\end{proof}

\begin{corollary}\label{cor:ConditionK}
If $E$ is a discrete graph satisfying condition (K), then every tracial state on $\cstar(E)$ is gauge invariant.
\end{corollary}

\begin{proof}
For $\alpha \in E^k$, $k \geq 1$, let $\delta_\alpha \in C_c(E^k)$ denote the indicator function of $\{\alpha\}$ and let $s_\alpha = \pi^k(\delta_\alpha) \in \cstar(E)$.  Similarly for $v \in E^0$, we let $p_v = \pi^0(\delta_v) \in \cstar(E)$.

Suppose $\tau$ is a tracial state on $\cstar(E)$ and let $\mu$ denote the induced invariant measure on $E$.  Assume $v \in E^0$ is the source of a cycle.  We claim there is an $n \geq 1$ and distinct cycles $\alpha, \beta \in E^n$ with $s(\alpha) = s(\beta) = v$.  There are distinct simple cycles $\alpha' \in E^k$ and $\beta' \in E^\ell$ with $s(\alpha') = s(\beta')$.  If $k = \ell$, set $n = k = \ell$, $\alpha = \alpha'$, and $\beta = \beta'$.  Otherwise, if $k \neq \ell$, define $n := k \ell$, $\alpha := (\alpha')^\ell \in E^n$ and $\beta := (\beta')^k \in E^n$.  If $\alpha_n \neq \beta_n$, then $\alpha$ and $\beta$ are distinct.  If $\alpha_n = \beta_n$, then the edge $\alpha_n = \alpha'_k = \beta'_\ell$ occurs exactly $\ell$ times in the path $\alpha$ and $k$ times in the path $\beta$ since $\alpha'$ and $\beta'$ are simple.  As $k \neq \ell$, we have $\alpha \neq \beta$.

Now since $\alpha, \beta \in E^n$ are distinct cycles, $s_\alpha^* s_\beta = 0$.  It follows that $s_\alpha s_\alpha^* + s_\beta s_\beta^* \leq p_v$ and $s_\alpha^* s_\alpha = s_\beta^* s_\beta = p_v$.  As $\tau$ is a tracial state, $0 \leq 2\tau(p_v) \leq \tau(p_v)$ and $\tau(p_v) = 0$.  So $v \notin \operatorname{supp}(\mu)$ and the result follows Proposition \ref{prop:NoCycles}.
\end{proof}

\section{Totally Disconnected Graphs and K-Theory}

In this section, we will show the gauge invariant tracial states on $\cstar(E)$ can be detected in the K-theory of $\cstar(E)$ when $E^0$ is totally disconnected.  First we recall the Pimsner-Voiculescu sequence for topological graphs given in \cite{Katsura:TGA1}.  For a topological graph $E$, view the Hilbert module $H(E)$ as a $C_0(E^0_\reg)$--$C_0(E^0)$ \cstar-correspondence by restricting scalars on the left.  Then taking the Kasparov product with the \cstar-correspondence $H(E)$ induces a morphism $[E] : \mathrm{K}^*(E^0_\reg) \rightarrow \mathrm{K}^*(E^0)$ on topological K-theory.  There is a natural six term exact sequence
\begin{equation}\label{eqn:PimsnerVoiculescuSequence}
  \begin{tikzcd}
   \mathrm{K}^0(E^0_\reg) \arrow{r}{\iota - [E]} & \mathrm{K}^0(E^0) \arrow{r}{\pi^0_*} & \mathrm{K}_0(\cstar(E)) \arrow{d} \\
   \mathrm{K}_1(\cstar(E)) \arrow{u} & \mathrm{K}^1(E^0) \arrow{l}{\pi^0_*} & \mathrm{K}^1(E^0_\reg) \arrow{l}{\iota - [E]}
  \end{tikzcd}
\end{equation}
When $E$ is a totally disconnected topological graph, the six term exact sequence takes a simpler form (Proposition \ref{prop:PimsnerVoiculescuSequence}).  First we introduce some notation.

For a totally disconnected space $X$, we view $\mathrm{K}^0(X)$ as the Grothendieck group of the finitely generated Hilbert modules over $C_0(X)$.  Given a compact open set $U \subseteq X$, the ideal $C(U) \unlhd C_0(X)$ is a finitely generated Hilbert module over $C_0(X)$.  The Hilbert modules of the form $C(U)$ generate $\mathrm{K}^0(X)$ as an abelian group.  Moreover, there is an isomorphism $\rho_X : \mathrm{K}^0(X) \rightarrow C_0(X, \mathbb{Z})$ given by $C(U) \mapsto \chi_U$ for each compact open set $U \subseteq X$, where $\chi_U$ is the indicator function of the set $U$.

If $X$ and $Y$ are totally disconnected spaces, $H$ is a Hilbert $C_0(X)$-module, and a representation $C_0(Y) \rightarrow \mathbb{K}(H)$ is given, then the \cstar-correspondence $H$ induces a morphism $[H] : \mathrm{K}^0(Y) \rightarrow \mathrm{K}^0(X)$ given by $[K] \mapsto [K \otimes_{C_0(Y)} H]$ for every finitely generated Hilbert $C_0(Y)$-module $K$.

Specializing to a totally disconnected topological graph $E$, there is an faithful representation $C_0(E^0_\reg) \rightarrow \mathbb{K}(H(E))$.  The morphism $[E] : \mathrm{K}^0(E^0_\reg) \rightarrow \mathrm{K}^0(E^0)$ is the morphism induced by the \cstar-correspondence $H(E)$.

\begin{proposition}\label{prop:PimsnerVoiculescuSequence}
For a totally disconnected topological graph $E$, there is an exact sequence
\[ \begin{tikzcd}
   0 \arrow{r} &\mathrm{K}_1(\cstar(E)) \arrow{r} & C_0(E^0_\reg, \mathbb{Z}) \arrow{r}{\iota - \psi} & C_0(E^0, \mathbb{Z}) \arrow{r} & \mathrm{K}_0(\cstar(E)) \arrow{r} & 0.
 \end{tikzcd} \]
where the map $\psi : C_0(E^0_\reg, \mathbb{Z}) \rightarrow C_0(E^0, \mathbb{Z})$ is given by
\[ \psi(f)(v) = \sum_{s(e) = v} f(r(e)) \]
for all $f \in C_0(E^0_\reg, \mathbb{Z})$ and $v \in E^0$.
\end{proposition}

\begin{proof}
Since $E^0$ and $E^0_\reg$ are totally disconnected, $\mathrm{K}^1(E^0_\reg) = \mathrm{K}^1(E^0) = 0$.  Hence in view of \eqref{eqn:PimsnerVoiculescuSequence}, it is enough to show
\begin{equation}\label{eqn:KTheorySquare}
  \begin{tikzcd} \mathrm{K}^0(E^0_\reg) \arrow{r}{[E]} \arrow{d}{\rho_{E^{^0}_\reg}} & \mathrm{K}^0(E^0) \arrow{d}{\rho_{E^{^0}}} \\ C_0(E^0_\reg, \mathbb{Z}) \arrow{r}{\psi} & C_0(E^0, \mathbb{Z}) \end{tikzcd}
\end{equation}
commutes.

Fix a compact open set $U \subseteq E^0_\reg$.  Then $C(U)$ is a singly generated Hilbert $C_0(E^0_\reg)$-module with generator $1_{C(U)}$.  Moreover, $r^{-1}(U)$ is a compact open subset of $E^1$ and hence is a Hilbert submodule of $H(E)$.  The map
\[ C(U) \otimes_{C_0(E^{^0}_\reg)} H(E) \rightarrow C(r^{-1}(U)) \qquad a \otimes f \mapsto (a \circ r)f \]
and the map
\[ C(r^{-1}(U)) \rightarrow C(U) \otimes_{C_0(E^{^0}_\reg)} H(E) \qquad g \mapsto 1_{C(U)} \otimes g \circ r \]
are inverse isomorphisms of Hilbert $C_0(E^0)$-modules.

Let $V_1, \ldots, V_n \subseteq E^1$ be a partition of $U$ into compact open sets such that $s|V_i$ is a homeomorphism onto $s(V_i)$ for each $i = 1, \ldots, n$, and let $\sigma_i : V_i \rightarrow s(V_i)$ denote the homeomorphism obtained by restricting $s$.  Then the map
\[ C(r^{-1}(U)) \rightarrow \bigoplus_{i=1}^n C(s(V_i)) \qquad f \mapsto (f \circ \sigma_i^{-1})_{i=1}^n \]
and the map
\[ \bigoplus_{i=1}^n C(s(V_i)) \rightarrow C(r^{-1}(U)) \qquad (g_i)_{i=1}^n \mapsto \sum_{i=1}^n g \circ \sigma_i \]
are inverse isomorphisms of Hilbert $C_0(E^0)$-modules.

Combining the last two paragraphs, we have
\[ [E](C(U)) = [C(U) \otimes H(E)] = [C(r^{-1}(U))] = \sum_{i=1}^n [C(s(V_i))] \]
in $\mathrm{K}^0(E^0)$.  Hence for $v \in E^0$,
\[ \rho_{E^{^0}}([E](C(U)))(v) = \sum_{i=1}^n \chi_{s(V_i)}(v) = \# ( r^{-1}(U) \cap s^{-1}(v) ) \]
where $\#S$ denotes the cardinality of a set $S$.  Similarly, for $v \in E^0$,
\[ \psi(\rho_{E^{^0}_\reg}(C(U)))(v) = \psi(\chi_U)(v) = \sum_{s(e) = v} \chi_U(r(e)) = \# ( r^{-1}(U) \cap s^{-1}(v) ). \]
Hence the diagram \eqref{eqn:KTheorySquare} commutes and this completes the proof.
\end{proof}

Given any \cstar-algebra $A$, the group $\mathrm{K}_0(A)$ has a natural order structure determined by the semigroup $\mathrm{K}_0(A)^+ \subseteq \mathrm{K}_0(A)$ consisting of all elements in $\mathrm{K}_0(A)$ defined by a matrix projection over $A$.  There is also a distinguished subset $\Sigma(A) \subseteq \mathrm{K}_0(A)^+$ called the \emph{scale} consisting of all elements in $\mathrm{K}_0(A)$ defined by projections in $A$.  If $A$ has an approximate unit consisting a projections, a \emph{state} on $\mathrm{K}_0(A)$ is a group morphism $\tau : \mathrm{K}_0(A) \rightarrow \mathbb{R}$ such that $\tau(x) \geq 0$ for all $x \in \mathrm{K}_0(A)^+$ and $\sup_{x \in \Sigma(A)} \tau(x) = 1$.  The collection of states on $\mathrm{K}_0(A)$ is denoted $S(\mathrm{K}_0(A))$.  Note that $S(\mathrm{K}_0(A))$ is convex and weak* compact.

For any \cstar-algebra $A$ with an approximate unit consisting of projections, there is a natural continuous affine map $T(A) \rightarrow S(\mathrm{K}_0(A))$ induced by restricting a tracial state to the set of projections.  Moreover, when $A$ is unital, Blackadar and R{\o}rdam have shown every state on $T(A) \rightarrow S(\mathrm{K}_0(A))$ is induced by a quasitrace on $A$ (see \cite{BlackadarRordam}).  This result also holds when $A$ has an approximate unit consisting of projections as can be shown by applying the unital result to the corners defined by these projections.  Moreover, when $A$ is exact, every quasitrace on $\cstar(E)$ is a trace as was shown by Haagerup in \cite{Haagerup} in the unital case and was extended to non-unital \cstar-algebras by Kirchberg in \cite{Kirchberg}.  Combining these results, when $A$ is a exact \cstar-algebra with an approximate unit consisting of projections, the canonical map $T(A) \rightarrow S(\mathrm{K}_0(A))$ is always surjective.  In particular this holds for $A = \cstar(E)$ when $E$ is a totally disconnected topological graph.

There is no known method for computing the order structure of $\mathrm{K}_0(\cstar(E))$ for a topological graph $E$, even when $E^0$ is totally disconnected. However, there are special cases.  When $E^0$ is discrete, the ordered $\mathrm{K}_0$-group is computed in \cite{Tomforde:OrderedKTheory}, and when $\cstar(E)$ is given by a dynamical system $\sigma : \mathbb{Z} \curvearrowright X$ with $X$ totally disconnected, the ordered $\mathrm{K}_0$-group is described in Theorem 5.2 of \cite{BoyleHandelman}.  In both cases, the order structure on $\mathrm{K}_0$ is determined by the Pimsner-Voiculescu sequence above.  It may be posible to give a similar description of the order structure for general topological graph algebras in the totally disconnected setting.  Computing the states on K-theory can be viewed as a first step towards such a result.  Indeed, for a simple, unital \cstar-algebra $A$, the order structure on the $\mathrm{K}_0$-group is completely determined, at least up to perforation, by the states.

Composting the map $T(\cstar(E)) \rightarrow S(\mathrm{K}_0(\cstar(E)))$ with the map obtained in \ref{thm:InvariantTrace} yields a continuous affine map $T(E) \rightarrow S(\mathrm{K}_0(\cstar(E)))$.  We will show this map is in fact bijective.

\begin{theorem}\label{thm:StatesOnK0}
Let $E$ be a topological graph with $E^0$ totally disconnected.  The canonical map $T(\cstar(E)) \rightarrow S(\mathrm{K}_0(\cstar(E)))$ restricts to an affine homeomorphism on $T(\cstar(E))^\gamma$.
\end{theorem}

\begin{proof}
It is clear that the map is continuous and affine.  As noted above, every state $f$ on $S(\mathrm{K}_0(\cstar(E)))$ lifts to a tracial state $\tau$ on $\cstar(E)$.  Define a tracial state $\tau_0$ on $\cstar(E)$ by
\[ \tau_0(a) = \int_{\mathbb{T}} \tau(\gamma_z(a)) \, dz \]
and note that $\tau_0$ is gauge invariant.  Fix a projection $p \in \mathbb{M}_n(\cstar(E))$.  For each $z \in \mathbb{T}$, there is a path of projections in $\mathbb{M}_n(\cstar(E))$ connecting $p$ and $\gamma_z(p)$ since $\gamma$ is a strongly continuous action and $\mathbb{T}$ is connected.  Hence $p$ and $\gamma_z(p)$ are unitarily equivalent in the unitization of $\mathbb{M}_n(\cstar(E))$ (see Proposition 4.3.3 in \cite{Blackadar:KTheory}, for example) and in particular, $\tau(p) = \tau(\gamma_z(p))$.  It follows that $\tau(p) = \tau_0(p)$ for every projection $p \in \mathbb{M}_n(\cstar(E))$.  In particular, $\tau_0$ also induces the state $f$ on $\mathrm{K}_0(\cstar(E))$.

Now suppose $\tau_1$ and $\tau_2$ are two gauge invariant tracial states on $\cstar(E)$ which agree on $\mathrm{K}_0(\cstar(E))$.  Then the tracial states $\tau_1 \circ \pi^0$ and $\tau_2 \circ \pi^0$ induce the same state $\mathrm{K}_0(C_0(E^0))$.  Since $E^0$ is totally disconnected, $C_0(E^0)$ is spanned by projections and hence $\tau_1 \circ \pi^0 = \tau_2 \circ \pi^0$.  Now for $k \in \mathbb{N}$ and $\xi, \eta \in C_c(E^k)$,
\[ \tau_1(\pi^k(\xi) \pi^k(\eta)^*) = \tau_1(\pi^0(\langle \eta, \xi \rangle)) = \tau_2(\pi^0(\langle \eta, \xi \rangle)) = \tau_2(\pi^k(\xi)\pi^k(\eta)^*) \]
since both $\tau_1$ and $\tau_2$ are tracial states.  For distinct $k, \ell \in \mathbb{N}$, $\xi \in C_c(E^k)$, and $\eta \in C_c(E^\ell)$,
\[ \tau_1(\pi^k(\xi) \pi^\ell(\eta)^*) = \tau_2(\pi^k(\xi) \pi^\ell(\eta)^*) = 0 \]
since both $\tau_1$ and $\tau_2$ are gauge invariant.  So $\tau_1 = \tau_2$.

We have shown the canonical map $T(\cstar(E))^\gamma \rightarrow S(\mathrm{K}_0(\cstar(E)))$ is continuous, affine, and bijective.  To show the inverse in continuous, suppose $(\tau_i)$ is a net of gauge invariant tracial states on $\cstar(E)$ and $\tau$ is a gauge invariant tracial state on $\cstar(E)$.  Let $f_i$ and $f$ denote the states on $\mathrm{K}_0(\cstar(E))$ induced by the $\tau_i$ and $\tau$.  If $f_i \rightarrow f$, then $f_i \circ \pi^0_* \rightarrow f \circ \pi^0_*$ on $\mathrm{K}_0(C_0(E^0))$.  Since $C_0(E^0)$ is spanned by its projections, $\tau_i \circ \pi^0 \rightarrow \tau \circ \pi^0$ weak*.  By Lemma \ref{lem:ConvergesOfTracesOnCuntzPimsnserAlgebras}, $\tau_i \rightarrow \tau$ weak* and this completes the proof.
\end{proof}

For an simple, weakly unperforated ordered group $K$ with a distinguished order unit, positive cone is determined by the states on $K$; for $x \in K$, $x > 0$ if and only if $f(x) > 0$ for each $f \in S(K)$ (see Theorem 6.8.5 in \cite{Blackadar:KTheory}).  We were not able to determine when $\mathrm{K}_0(\cstar(E))$ is weakly unperforated for a minimal, totally disconnected topological graph $E$.  However, we can still describe the positive cone in $\mathrm{K}$-theory up to perforation.

\begin{corollary}\label{cor:OrderedKTheory}
Suppose a minimal topological graph in the sense of \cite[Definition 8.8]{Katsura:TGA3} such that $E^0$ is totally disconnected and compact.  Let $x \in \mathrm{K}_0(\cstar(E))$ be given and let $a\in C_0(E^0, \mathbb{Z})$ be such that $\pi^0(a) = x$.  Then $n x \geq 0$ for some $n \geq 1$ if and only if $\int a \, d\mu \geq 0$ for all $\mu \in T(E)$.
\end{corollary}

\begin{proof}
Since $E^0$ is compact, $\cstar(E)$ is unital.  By Theorem 8.12 in \cite{Katsura:TGA3}, either $\cstar(E)$ is simple or $E$ is generated by a cycle.  If $E$ is generated by a cycle, then $\cstar(E)$ is Morita equivalent to $C(\mathbb{T})$.  In either case, the ordered abelian group $\mathrm{K}_0(\cstar(E))$ is simple.  Theorem \ref{thm:StatesOnK0} implies $f(x) \geq 0$ for all $f \in S(\mathrm{K}_0(\cstar(E)))$ if and only if $\int a \, d\mu \geq 0$ for all $\mu \in T(E)$.  The result follows from (the proof of) Theorem 6.8.5 in \cite{Blackadar:KTheory}.
\end{proof}


\begin{thebibliography}{99}
\bibitem{BoyleHandelman} \textsc{M. Boyle, D. Handelman}, Orbit equivalence, flow equivalence and ordered cohomology, \textit{Israel J. Math.} \textbf{95} (1996), 169-–210.
\bibitem{Blackadar:KTheory} \textsc{B. Blackadar}, \textit{K-theory for operator algebras, 2nd ed.}, Mathematical Sciences Research
Institute Publications, vol. 5, Cambridge University Press, Cambridge 1998.
\bibitem{BlackadarRordam} \textsc{B. Balckadar, M. R{\o}rdam}, Extending states on preordered semigroups and the existence
of quasitraces on \cstar-algebras, \textit{J. Algebra}, \textbf{152} (1992), 240–-247.
\bibitem{Dixmier} \textsc{J. Dixmier}, \textit{C*-algebras}, North-Holland Publishing Company, Amersterdam, 1977.
\bibitem{Folland} \textsc{G. Folland}, \textit{Real analysis: modern techniques and their applications, 2nd. ed.}, Wiley, New York 1999.
\bibitem{Haagerup} \textsc{U. Haagerup}, Quasitraces on exact \cstar-algebras are traces, \textit{C. R. Math. Acad. Sci. Soc. R. Can.}, \textbf{36} (2014), 67--92.
\bibitem{Kirchberg} \textsc{E. Kirchberg}, On the existence of traces on exact stably projectionless simple \cstar-algebras, \textit{Operator Algebras and their Applications (P. A. Fillmore and J. A. Mingo, eds.) Fields Institute Communications}, vol. 13, Amer. Math. Soc. 1995, 171–-172.
\bibitem{Katsura:TGA1} \textsc{T. Katsura}, A class of \cstar-algebras generalizing both graph algebra and homeomorphism \cstar-algebras I, fundamental results, \textit{Trans. Amer. Math. Soc.} \textbf{356} (2004), 4287-–4322.
\bibitem{Katsura:TGA2} \textsc{T. Katsura}, A class of \cstar-algebras generalizing both graph algebra and homeomorphism \cstar-algebras II, examples, \textit{Int. J. Math.} \textbf{17} (2006), 791–-833.
\bibitem{Katsura:TGA3} \textsc{T. Katsura}, A class of \cstar-algebras generalizing both graph algebra and homeomorphism \cstar-algebras III, ideal structures, \textit{Ergodic Theory Dynam. Systems} \textbf{26} (2006), 1805-–1854.
\bibitem{Katsura:TGA4} \textsc{T. Katsura}, A class of \cstar-algebras generalizing both graph algebra and homeomorphism \cstar-algebras IV, pure infinitenss, \textit{J. Funct. Anal.} \textbf{254} (2008), 1161-–1187.
\bibitem{KumjianLi}\textsc{A. Kumjian, H. Li}, Twisted topological graph algebras II, twisted groupoid \cstar-algebras, preprint. (arXiv:1507.04449 [math.OA]).
\bibitem{Raeburn} \textsc{I. Raeburn}, \textit{Graph Algebras},  CBMS Regional Conference Series in Mathematics \textbf{103} Published for the Conference Board of the Mathematical Sciences, Washington D.C.~by the AMS, Providence, RI (2005).
\bibitem{Renault} \textsc{J. Renault}, Cuntz-like algebras, \textit{Operator theoretic models (Timi\c{s}oara, 1998)}, 371--386, Theta Found., Bucharest, 2000.
\bibitem{Rordam} \textsc{M. R{\o}rdam}, A simple \cstar-algebra with a finite and an infinite projection, \textit{Acta Math.}, \textbf{191} (2003), 109--142.
\bibitem{Tomforde:OrderedKTheory} \textsc{M. Tomforde}, The ordered $\mathrm{K}_0$-group of a graph \cstar-algebras, \textit{C. R. Math. Acad. Sci. Soc. R. Can.}, \textbf{25} (2003) 19--25.
\bibitem{Yeend1} \textsc{T. Yeend}, Groupoid models for the \cstar-algebras of topological higher rank graphs, \textit{J. Operator Theory}, \textbf{57} (2007), 95--120.
\bibitem{Yeend2} \textsc{T. Yeend}, Topoological higher-rank graphs and the \cstar-algebras of toplogical 1-graphs, \textit{Contemp. Math.} 414, Operator theory, operatator algebras, and applications, 231--244, Amer. Math. Soc., Providence, RI, 2006.
\bibitem{Webster} \textsc{S.B.G. Webster}, The path space of a directed graph, \textit{Proc. Amer. Math. Soc.} \textbf{142} (2014), 213--225.
\end{thebibliography}
\end{document}